\documentclass[11pt]{article}
\usepackage[pdftex]{graphicx}
\usepackage{bm}
\usepackage{amsmath,amsthm,amssymb,url}
\usepackage{multirow,bigdelim}
\usepackage{amscd}

\usepackage{tikz}
\usetikzlibrary{arrows}

\newtheorem{thm}{Theorem}[section]
\newtheorem{theorem}[thm]{Theorem}
\newtheorem{proposition}[thm]{Proposition}
\newtheorem{lemma}[thm]{Lemma}
\newtheorem{definition}{Definition}

\newtheorem{corollary}[thm]{Corollary}

\newtheorem{conj}{Conjecture}

\newcommand{\te}{\tilde{e}}

\newcommand{\tv}{\tilde{v}}
\newcommand{\bp}{{\bm p}}

\newcommand{\bmm}{{\bm m}}

\newcommand{\bome}{{\bm \omega}}

\title{Infinitesimal Rigidity of Symmetric Bar-Joint Frameworks}
\author{Bernd Schulze \thanks{
Department of Mathematics and Statistics,
University of Lancaster,
Lancaster
LA1 4YF, United Kingdom
(\texttt{b.schulze@lancaster.ac.uk}).
}
\and
Shin-ichi Tanigawa\thanks{
  Research Institute for Mathematical Sciences, Kyoto University, Kyoto 606-8502 Japan
(\texttt{tanigawa@kurims.kyoto-u.ac.jp}).
  Supported by JSPS Grant-in-Aid for Young Scientist (B), 24740058.}
}

\begin{document}

\renewcommand{\thefootnote}{ }

\maketitle

\begin{abstract}

We propose new symmetry-adapted rigidity matrices to analyze the infinitesimal
rigidity of arbitrary-dimensional bar-joint frameworks with Abelian point group symmetries. These matrices define
new symmetry-adapted rigidity matroids on group-labeled quotient graphs. Using these new tools,
we establish combinatorial characterizations of infinitesimally rigid
two-dimensional bar-joint frameworks whose joints are positioned as generic as possible
subject to the symmetry constraints imposed by a reflection, a half-turn or a
three-fold rotation in the plane. For bar-joint frameworks which are generic with respect
to any other cyclic point group in the plane, we provide a number of necessary conditions for infinitesimal rigidity.

\end{abstract}

%%%%%%%%%%%%%%%%%%%%%%%%%%%%%%%%%%%%%%%%%%%%%%%%%%%%%%%%%%%%%%%%%%%%%%%%%%%%%%%%%%%%%%%%%%%%%%%%%%%%%%%%%%%%%%%%%%%%%%%%%%%%%%%%%%%%%%%%%%

\section{Introduction}
\label{sec:intro}

A $d$-dimensional bar-joint framework is a straight-line realization of a finite simple graph $G$ in Euclidean $d$-space. Intuitively, we think of a bar-joint framework as a collection of fixed-length bars (corresponding to the edges of $G$) which are connected at their ends by joints (corresponding to the vertices of $G$) that allow bending in any direction of $\mathbb{R}^d$. Such a  framework is said to be rigid if there exists no non-trivial continuous bar-length preserving motion of the framework vertices, and is said to be flexible otherwise (see \cite{W1} for basic definitions and background).

The theory of generic rigidity seeks to characterize the graphs which form rigid frameworks for all generic (i.e., almost all) realizations of the vertices in Euclidean $d$-space. For $d=2$, this problem was first solved by
Laman~\cite{Lamanbib} in 1970:
% celebrated result from 1970 which provides a characterization of rigid generic frameworks in the plane in terms of some simple counting conditions
 Laman proved that a generic two-dimensional bar-joint framework is minimally rigid if and only if
the underlying graph $G$ satisfies $|E(G)|=2|V(G)|-3$ and $|E(G')|\leq 2|V(G')|-3$ for any subgraph
$G'$ of $G$ with $|V(G')|\geq 2$, where $V(H)$ and $E(H)$ denote the set of vertices and
the set of edges of a graph $H$, respectively.
%Several further equivalent characterizations of  rigid generic frameworks in the plane have been established since then (see \cite{tay,lovyem}, for example).
% result from 1970 (see Theorem \ref{...}) \cite{Lamanbib}, and several further characterizations of  rigid generic frameworks in the plane have been established since then \ref{....}.
 For dimensions $d\geq 3$, however, the analogous questions remain long-standing open problems, although there exist some significant partial results \cite{W1}. 

The theory of rigid and flexible frameworks has a wide variety of practical applications in many areas of science, engineering and design, where frameworks serve as a suitable mathematical model for various kinds of physical structures, mechanical gadgets (such as linkages or robots), sensor networks, biomolecules, etc. Since many of these structures exhibit non-trivial symmetries, it is natural to explore the impact of symmetry on the rigidity and flexibility properties of frameworks. Over the last decade, this research area has gained an ever increasing attention in both the mathematical community and in the applied sciences.
Two separate fundamental research directions can be identified:
\begin{enumerate}
\item Forced symmetry: The framework starts in a symmetric position and must maintain this symmetry throughout its motion.

\item Incidental symmetry: The framework starts in a symmetric position, but may move in unrestricted ways.
\end{enumerate}

Over the last few years, significant progress has been made in the rigidity analysis of forced-symmetric frameworks \cite{MT1,MT2,LT,jkt,BSWW,tan}.  A key motivation for this research is that for symmetry-generic frameworks (that is, for frameworks which are as generic as possible subject to the given symmetry constraints), the existence of a non-trivial symmetric infinitesimal motion also guarantees the existence of a non-trivial finite (i.e., continuous) symmetry-preserving motion  of the framework \cite{BS2}.  To simplify  the symmetry-forced rigidity analysis of a symmetric framework  a symmetric analog of the rigidity matrix, called the orbit rigidity matrix, was recently established in \cite{BSWW}. In particular, this matrix was used in  \cite{jkt} to formulate combinatorial characterizations of symmetry-forced rigid  symmetry-generic frameworks in terms of Henneberg-type construction moves on gain graphs (group-labeled graphs), for all rotational groups ${\cal C}_n$  and for all dihedral groups ${\cal C}_{nv}$ with odd $n$ in the plane.

In contrast, for the more general question of how to analyze the rigidity properties of an incidentally symmetric framework, there has not been any major progress in the last few years.
This paper proposes a systematic way to analyze this general case.
The state of the art in this research area is as follows.

The most fundamental result concerning the rigidity of symmetric frameworks is that the rigidity matrix of a framework with non-trivial point group $\Gamma$ can be transformed into a block-decomposed form so that each block corresponds to an irreducible representation of $\Gamma$. This goes back to an observation of Kangwai and Guest \cite{KG2}, and was proved rigorously in \cite{BS2,owen}.
%(We also include a short proof of this result in Section ...).
Note that the submatrix block which corresponds to the trivial irreducible representation of $\Gamma$ describes the forced-symmetric rigidity properties of the framework~\cite{BSWW}.
Using this block-decomposition of the rigidity matrix, necessary conditions for a symmetric bar-joint framework  to be  isostatic (i.e., minimally infinitesimally rigid) in $\mathbb{R}^d$ have been derived in \cite{FGsymmax,cfgsw}.
%For $d=2$, these conditions imply that an isostatic bar-and-joint framework in the plane must not only satisfy the Laman conditions, but also some very simply stated restrictions on the number of joints and bars that are `fixed' by various symmetry operations of the framework. In particular, it turns out that the point group of a $2$-dimensional isostatic framework must be one of the six groups $\mathcal{C}_1$, $\mathcal{C}_2$,  $\mathcal{C}_3$, $\mathcal{C}_s$, $\mathcal{C}_{2v}$ or $\mathcal{C}_{3v}$, where $\mathcal{C}_1$ is the trivial group, $\mathcal{C}_2$ and  $\mathcal{C}_3$ are the groups generated by a half-turn and a three-fold rotation, respectively, and $\mathcal{C}_{2v}$ and $\mathcal{C}_{3v}$ are the dihedral groups of order 4 and 6.
%% cannot contain a rotation by an angle of $2\frac{2\pi}{k}$ for $k>3$.
%For $3$-dimensional frameworks, all point groups are possible. However, restrictions to the placement
%of structural components still apply \cite{cfgsw}.

%In \cite{cfgsw} it was conjectured that the Laman conditions, together with the additional conditions on the number of
%fixed structural components, are both necessary and sufficient for a $2$-dimensional symmetry-generic framework to be isostatic.
%This result was confirmed  for the groups $\mathcal{C}_2$, $\mathcal{C}_3$ and $\mathcal{C}_s$ in \cite{BS3,BS4}, but it remains open for both of the dihedral groups.

In \cite{cfgsw} the necessary conditions were conjectured to be sufficient for
$2$-dimensional symmetry-generic frameworks to be isostatic.
This was confirmed  for the groups $\mathcal{C}_2$, $\mathcal{C}_3$ and $\mathcal{C}_s$ in \cite{BS3,BS4}, but it remains open for the dihedral groups.

However, note that in order to obtain combinatorial characterizations of symmetry-generic infinitesimally rigid frameworks in the plane these symmetrized Laman-type results are only of limited use since, by the conditions derived in \cite{cfgsw}, a symmetric infinitesimally rigid framework usually does not contain an isostatic subframework on the same vertex set with the same symmetry. For example, it turns out that there does not exist an isostatic framework in the plane with point group $\mathcal{C}_2$ or $\mathcal{C}_s$, where the group acts freely on the edges of the framework (see Figure \ref{fig:noisosubfw}) \cite{cfgsw}. Moreover, there does not exist \emph{any} isostatic framework in the plane with $k$-fold rotational symmetry, for $k>3$  \cite{cfgsw}.

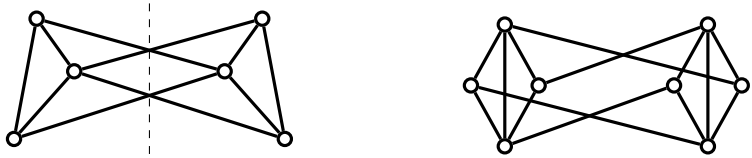
\begin{figure}[htp]
\begin{center}
  \begin{tikzpicture}[very thick,scale=1]
\tikzstyle{every node}=[circle, draw=black, fill=white, inner sep=0pt, minimum width=5pt];
     \path (-1,0) node (p1)  {} ;
    \path (1,0) node (p2)  {} ;
    \path (-1.5,0.7) node (p3)  {} ;
     \path (1.5,0.7) node (p4)  {} ;
     \path (-1.8,-0.9) node (p5)  {} ;
     \path (1.8,-0.9) node (p6)  {} ;
\draw (p1)  --  (p3);
        \draw (p1)  --  (p5);
        \draw (p1)  --  (p4);
        \draw (p6)  --  (p1);

        \draw (p2)  --  (p4);
        \draw (p2)  --  (p6);
        \draw (p2)  --  (p3);
        \draw (p5)  --  (p2);

        \draw (p5)  --  (p3);
        \draw (p4)  --  (p6);
        \draw[dashed,thin] (0,-1.1)  --  (0,1);
              \end{tikzpicture}
        \hspace{2cm}
         \begin{tikzpicture}[very thick,scale=0.9]
\tikzstyle{every node}=[circle, draw=black, fill=white, inner sep=0pt, minimum width=5pt];
     \path (-1,0) node (p1)  {} ;
    \path (1,0) node (p2)  {} ;
    \path (-2,0) node (p3)  {} ;
    \path (2,0) node (p4)  {} ;
       \path (-1.5,0.9) node (p5)  {} ;
    \path (1.5,-0.9) node (p6)  {} ;
    \path (-1.5,-0.9) node (p7)  {} ;
    \path (1.5,0.9) node (p8)  {} ;

        \draw (p1)  --  (p5);
        \draw (p5)  --  (p3);
        \draw (p3)  --  (p7);
        \draw (p1)  --  (p7);
        \draw (p5)  --  (p7);

        \draw (p2)  --  (p6);
        \draw (p6)  --  (p4);
        \draw (p4)  --  (p8);
        \draw (p2)  --  (p8);
        \draw (p6)  --  (p8);

        \draw (p1)  --  (p8);
        \draw (p2)  --  (p7);
        \draw (p5)  --  (p4);
        \draw (p6)  --  (p3);

        \end{tikzpicture}

     \caption{Infinitesimally rigid symmetric frameworks in $\mathbb{R}^2$ with respective point groups $\mathcal{C}_s$ and $\mathcal{C}_2$  which do not contain a spanning isostatic subframework with the same symmetry.}
\end{center}
\label{fig:noisosubfw}
\end{figure}

In this paper, we establish several new results concerning the infinitesimal rigidity of (`incidentally') symmetric frameworks.
First, for any Abelian point group  $\Gamma$  which acts freely on the vertices of a $d$-dimensional framework, we extend the concept of the orbit rigidity matrix described in \cite{BSWW} and show how to construct an `anti-symmetric' orbit rigidity matrix
for each of the irreducible representations $\rho_j$ of $\Gamma$ (see Section \ref{sec:block}).
 These `anti-symmetric' orbit rigidity matrices are equivalent to their corresponding submatrix blocks in the block-decomposed rigidity matrix, but their entries can explicitly be derived in a transparent fashion.

For the reflection group $\mathcal{C}_s$ and for the rotational groups  $\mathcal{C}_2$ and   $\mathcal{C}_3$, we then use these orbit rigidity matrices in combination with Henneberg-type inductive construction moves on their corresponding gain graphs to establish combinatorial characterizations of symmetry-generic frameworks in $\mathbb{R}^2$ which do not have a non-trivial
$\rho_j$-symmetric infinitesimal motion. Taken together, these results lead to the desired  combinatorial characterizations of infinitesimally rigid symmetry-generic frameworks for these groups (see Sections~\ref{sec:gain_sparsity} and \ref{sec:characterization}).

 For the other cyclic groups $\mathcal{C}_k$, $k>3$, we provide a number of necessary conditions for infinitesimal rigidity, and we also offer some conjectures.

Finally, in Section~\ref{sec:ext}, we briefly discuss some further applications of our tools and methods and outline some directions for future developments.

%%%%%%%%%%%%%%%%%%%%%%%%%%%%%%%%%%%%%%%%%%%%%%%%%%%%%%%%%%%%%%%%%%%%%%%%%%%%%%%%%%%%%%%%%%%%%%%%%%%%%%%%%%%%%%%%%%

\section{Rigidity of bar-joint frameworks}

For a finite  graph $G$, we denote the vertex set of $G$ by $V(G)$ and the edge set of $G$ by $E(G)$.
%Two vertices $u \ne v$ of $G$ are said to be \emph{adjacent} if $\{u,v\}\in E(G)$, and \emph{independent} otherwise.
A \emph{bar-joint framework}  (or simply a {\em framework}) in $\mathbb{R}^{d}$ is a pair $(G,\bp)$,
where $G$ is a simple graph and  $\bp:V(G)\to \mathbb{R}^d$ is a map such that $\bp(u) \neq \bp(v)$
for all $\{u,v\} \in E(G)$.
%We also say that $(G,p)$ is a $d$-dimensional \emph{realization} of the \emph{underlying graph} $G$ \cite{gss, W1}.
For $v\in V(G)$, we say that $\bp(v)$ is the \emph{joint} of $(G,\bp)$ corresponding to $v$,
and for $e=\{u,v\}\in E(G)$, we say that the line segment between $\bp(u)$ and $\bp(v)$ is
the \emph{bar} of $(G,\bp)$ corresponding to $e$.
For simplicity, we shall denote $\bp(v)$ by $p_v$ for $v\in V(G)$.
%Note that in this paper, we will only consider frameworks in the plane, i.e., $d$ will always be equal to 2.

An \emph{infinitesimal motion} of a framework $(G,\bp)$ in $\mathbb{R}^d$
is a function $\bmm: V(G)\to \mathbb{R}^{d}$ such that
\begin{equation}
\label{infinmotioneq}
\langle p_u-p_v, m_u-m_v\rangle =0 \quad\textrm{ for all } \{u,v\} \in E(G)\textrm{,}
\end{equation}
where $m_v=\bmm(v)$ for each $v$.

An infinitesimal motion $\bmm$ of $(G,\bp)$ is a \emph{trivial infinitesimal motion}
if there exists a skew-symmetric matrix $S$
and a vector $t$ such that $\bmm(v)=S\bp(v)+t$ for all $v\in V(G)$.
Otherwise $\bmm$ is called an \emph{infinitesimal flex} (or \emph{non-trivial infinitesimal motion}) of $(G,\bp)$.
 $(G,\bp)$ is \emph{infinitesimally rigid} if every infinitesimal motion of $(G,\bp)$ is trivial.
Otherwise $(G,\bp)$ is said to be \emph{infinitesimally flexible} \cite{W1}.

These definitions are motivated by the fact that if $(G,\bp)$ is infinitesimally rigid, then $(G,\bp)$ is rigid
in the sense that every continuous deformation of $(G,\bp)$ which preserves the edge lengths
$\|p_i-p_j\|$ for all $\{i,j\}\in E(G)$, must preserve the distances $\|p_s-p_t\|$
for all pairs of vertices $s$ and $t$ of $G$.

A key tool to study the infinitesimal rigidity properties of a $d$-dimensional framework $(G,\bp)$
is the rigidity matrix of $(G,\bp)$.
For a vector $x\in \mathbb{R}^d$, we denote the $k^{th}$ component of $x$ by $(x)_{k}$.
The rigidity matrix $R(G,\bp)$ is a $|E(G)|\times d|V(G)|$ matrix associated with the system of linear equations
(\ref{infinmotioneq}) with respect to $\bmm$, in which each row is associated with an edge
and consecutive $d$ columns are associated with a vertex as follows,
\begin{displaymath} \bordermatrix{& & & & u & & & & v & & & \cr & & & &  & & \vdots & &  & & &
\cr e=\{u,v\} & 0 & \ldots &  0 & (p_{u}-p_{v}) & 0 & \ldots & 0 & (p_{v}-p_{u}) &  0 &  \ldots&  0 \cr & & & &  & & \vdots & &  & & &
}
\textrm{,}\end{displaymath}
where, for each edge $\{u,v\}\in E(G)$, $R(G,\bp)$ has the row with
$(p_{u}-p_{v})_{1},\ldots,(p_{u}-p_{v})_{d}$ in the columns associated with $u$,
$(p_{v}-p_{u})_{1},\ldots,(p_{v}-p_{u})_{d}$ in the columns associated with $v$,
and $0$ elsewhere \cite{W1}.

Throughout the paper, for a finite set $S$ and a finite dimensional vector space $W$ over some field,
the set of all functions $f:S\rightarrow W$ is denoted  by $W^S$ or by $\bigoplus_{s\in S} W$ (taking copies of $W$).
%It is often useful to identify $p$ with a vector in $\mathbb{R}^{dn}$ by using the order on $V(G)$. In this case we also refer to $p$ as a \emph{configuration} of $n$ points in $\mathbb{R}^{d}$.
Then $R(G,\bp)$ is regarded as a linear map from $(\mathbb{R}^d)^{V(G)}$ to $\mathbb{R}^{E(G)}$.
Note that $\bmm\in (\mathbb{R}^d)^{V(G)}$ is an infinitesimal motion if and only if
$R(G,\bp) \bmm=0$,
which means that  the kernel of the rigidity matrix $R(G,\bp)$ is
the space of all infinitesimal motions of $(G,\bp)$.
It is well known that a framework $(G,\bp)$ in $\mathbb{R}^d$ with $n=|V(G)|$ is infinitesimally rigid if and
only if either the rank of its associated rigidity matrix $R(G,\bp)$ is precisely
$dn-\binom{d+1}{2}$, or $G$ is a complete graph $K_n$ and the points $p_i$, $i=1,\ldots, n$, are affinely independent \cite{asiroth}.

A \emph{self-stress} of a framework $(G,\bp)$ is a function  $\bome:E(G)\to \mathbb{R}$ such that at each joint
$p_u$ of $(G,\bp)$ we have
\begin{displaymath}
\sum_{v :\{u,v\}\in E(G)}\omega_{uv}(p_{u}-p_{v})=0 \textrm{,}
\end{displaymath}
where $\omega_{uv}$ denotes $\bome(\{u,v\})$ for all $\{u,v\}\in E(G)$.
Note that $\bome\in \mathbb{R}^{E(G)}$ is a self-stress if and only if
$R(G,\bp)^{\top}\bome=0$.
In structural engineering, the self-stresses are also called \emph{equilibrium stresses} as they record  tensions and compressions in the bars balancing at each vertex.

 If $(G,\bp)$ has a non-zero self-stress, then $(G,\bp)$ is said to be \emph{dependent}
(since in this case there exists a linear dependency among the row vectors of $R(G,\bp)$).
Otherwise, $(G,\bp)$ is said to be \emph{independent}.
A framework which is both independent and infinitesimally rigid is called \emph{isostatic} \cite{W1}.

A $d$-dimensional framework $(G,\bp)$ with $n$ vertices is called \emph{generic} if the
coordinates of $\bp$ are algebraically independent over $\mathbb{Q}$, i.e., if there does
not exist a polynomial $h(x_1,\ldots, x_{dn})$ with rational coefficients such that
$h((p_1)_1\ldots,(p_n)_d)=0$. Note that the set of all generic realizations of $G$ is a dense, but not an open subset of $\mathbb{R}^{dn}$.

We say that $(G,\bp)$ is \emph{regular}
if the rigidity matrix $R(G,\bp)$ has maximal rank among all realizations of $G$. It is easy to see that the set of all regular realizations of $G$  is a dense and open subset of $\mathbb{R}^{dn}$ which contains the set of all generic realizations of $G$ \cite{asiroth,W1}.

  It is well known that for regular frameworks (and hence also for generic frameworks), infinitesimal rigidity is purely combinatorial, and hence a property of the underlying graph. Thus, we say that a graph $G$ is \emph{$d$-rigid ($d$-independent, $d$-isostatic)} if $d$-dimensional regular realizations of $G$ are infinitesimally rigid (independent, isostatic).

%In 1970, G. Laman proved the following combinatorial characterization of generically $2$-isostatic graphs.

%\begin{theorem}[Laman, 1970]\label{thm:laman}
%\label{lamantheorem}\cite{Lamanbib}
%A graph $G$ with $|V(G)|\geq 2$ is generically 2-isostatic if and only if
%\begin{itemize}
%\item[(i)] $|E(G)|=2|V(G)|-3$;
%\item[(ii)]$|E(H)|\leq 2|V(H)|-3$ for all $H\subseteq G$ with $|V(H)|\geq 2$.
%\end{itemize}
%\end{theorem}

%%%%%%%%%%%%%%%%%%%%%%%%%%%%%%%%%%%%%%%%%%%%%%%%%%%%%%%%%%%%%%%%%%%%%%%%%%%%%%%%%%%%%%%%%%%%%%%%%%%%%%%%%%%%%%%%%%%%%%%%%%%

\section{Rigidity of symmetric bar-joint frameworks}
In this subsection, we review some recent approaches for analyzing the rigidity of symmetric frameworks.
First, we  introduce gain graphs, which turn out to be useful tools for describing the underlying combinatorics
of symmetric frameworks.
We then provide precise definitions of symmetric graphs and symmetric frameworks, and then explain
the block-diagonalization of rigidity matrices.

\subsection{Gain graphs}
Let $H$ be a directed graph which may contain multiple edges and loops, and let $\Gamma$ be a group.
A {\em $\Gamma$-gain graph} (or $\Gamma$-labeled graph)  is a pair $(H,\psi)$ in which each edge is associated with an
element of $\Gamma$ via a {\em gain function} $\psi:E(H)\rightarrow \Gamma$.
See Figure \ref{gaingraph}(b) for an example.
A gain graph is a directed graph, but
its orientation is used only for the reference of the gains.
That is, we can change the orientation of each edge as we like
by imposing the  property on $\psi$ that if an edge has gain $g$ in one direction,
then it has gain $g^{-1}$ in the other direction.

%%%%%%%%%%%%%%%%%%%%%%%%%%%%%%%%%%%%%%%%%%%%%%%

\subsection{Symmetric graphs}
\label{subsec:symmetric_graphs}

Let $G$ be a finite simple graph.
An {\em automorphism} of $G$ is a permutation $\pi:V(G)\rightarrow V(G)$ such that
$\{u,v\}\in E(G)$ if and only if $\{\pi(u),\pi(j)\}\in E(G)$.
The set of all automorphisms of $G$ forms a subgroup of the symmetric group on $V(G)$,
known as the {\em automorphism group} ${\rm Aut}(G)$ of $G$.
An {\em action} of a group $\Gamma$ on $G$ is a group homomorphism $\theta:\Gamma \rightarrow {\rm Aut}(G)$.
An action $\theta$ is called {\em free} on $V(G)$ (resp., $E(G)$)
if  $\theta(\gamma)(v)\neq v$ for any $v\in V(G)$ (resp., $\theta(\gamma)(e)\neq e$ for any $e\in E(G)$) and
any non-identity $\gamma\in \Gamma$.
We say that a graph $G$ is {\em $\Gamma$-symmetric} (with respect to $\theta$)
if $\Gamma$ acts on $G$ by $\theta$.
Throughout the paper, we only consider the case when $\theta$ is free on $V(G)$, and we
omit to specify the action $\theta$, if it is clear from the context.
We then denote $\theta(\gamma)(v)$ by  $\gamma v$.

%Note that, for any $g\in {\cal S}$ and $u,v\in V$,
%\begin{equation}
%\label{eq:symmetry_vertex}
%\{u,v\}\in E(\Gamma) \ \Longleftrightarrow \ \{gu,gv\}\in E(\Gamma),
%\end{equation}
%and hence there is an induced  action of ${\cal S}$ on $E(\Gamma)$ defined by $g\cdot \{u,v\}=\{gu,gv\}$.

For a $\Gamma$-symmetric graph $G$,
the {\em quotient graph} $G/\Gamma$ is a multigraph whose vertex set is the set $V(G)/\Gamma$
of vertex orbits and whose edge set is the set $E(G)/\Gamma$ of edge orbits.
An edge orbit may be represented by a loop in $G/\Gamma$.

Several distinct graphs may have the same quotient graph.
However, if we assume that the underlying action is free on $V(G)$,
then a gain labeling  makes the relation one-to-one.
To see this, we arbitrarily choose a vertex $v$ as a representative vertex from each vertex orbit.
Then each orbit is of the form $\Gamma v=\{gv\mid g\in \Gamma\}$.
If the action is free,
an edge orbit connecting $\Gamma u$ and $\Gamma v$ in $G/\Gamma$
can be written as $\{\{gu,ghv\}\mid g\in \Gamma \}$
for a unique $h\in \Gamma$.
We then orient the edge orbit from $\Gamma u$ to $\Gamma v$ in $G/\Gamma$
and assign to it the gain $h$.
In this way, we obtain {\em the quotient $\Gamma$-gain graph}, denoted by $(G/\Gamma,\psi)$.
$(G/\Gamma,\psi)$ is unique up to choices of representative vertices.
Figure \ref{gaingraph} illustrates an example, where $\Gamma$ is the reflection group $\mathcal{C}_s$.

\begin{figure}[htp]
\begin{center}
  \begin{tikzpicture}[very thick,scale=1]
\tikzstyle{every node}=[circle, draw=black, fill=white, inner sep=0pt, minimum width=5pt];
     \path (-1,0) node (p1)  [label = left: $v_1$]{} ;
    \path (1,0) node (p2) [label = right: $v_4$]{}  ;
    \path (-1.8,0.8) node (p3) [label = left: $v_2$]{} ;
     \path (1.8,0.8) node (p4)  [label = right: $v_5$]{} ;
     \path (-1.8,-0.8) node (p5)[label = left: $v_3$]{} ;
     \path (1.8,-0.8) node (p6) [label = right: $v_6$]{};
\draw (p1)  --  (p3);
        \draw (p1)  --  (p5);
              \draw (p6)  --  (p1);

        \draw (p2)  --  (p4);
        \draw (p2)  --  (p6);
        \draw (p4)  --  (p3);
        \draw (p5)  --  (p2);

        \draw (p5)  --  (p3);
        \draw (p4)  --  (p6);
        \draw[dashed,thin] (0,-1)  --  (0,1.2);

           \node [rectangle, draw=white, fill=white] (b) at (0,-1.5) {(a)};
        \end{tikzpicture}
        \hspace{2cm}
         \begin{tikzpicture}[very thick,scale=0.9]
\tikzstyle{every node}=[circle, draw=black, fill=white, inner sep=0pt, minimum width=5pt];
     \path (-1.2,0.15) node (p1) [label = right: $v_1$]{} ;
       \path (-2,0.95) node (p3) [label = left: $v_2$]{} ;
      \path (-2,-0.65) node (p5) [label = left: $v_3$]{} ;
      \draw (p1)  --  (p3);
        \draw (p1)  --  (p5);
        \draw (p3)  --  (p5);
        \path
(p3) edge [loop above,->, >=stealth,shorten >=2pt,looseness=26] (p3);
\path
(p5) edge [->,bend right=22] (p1);
\node [rectangle, draw=white, fill=white] (b) at (-1.25,-0.45) {$s$};
\node [rectangle, draw=white, fill=white] (b) at (-1.65,1.25) {$s$};
\node [rectangle, draw=white, fill=white] (b) at (-1.6,-1.5) {(b)};
        \end{tikzpicture}
         \caption{A $\mathcal{C}_s$-symmetric graph (a) and its quotient gain graph (b), where $\mathcal{C}_s=\{id,s\}$. For simplicity, we omit the direction and the label of every edge with gain $id$.}
\end{center}
    \label{gaingraph}
\end{figure}
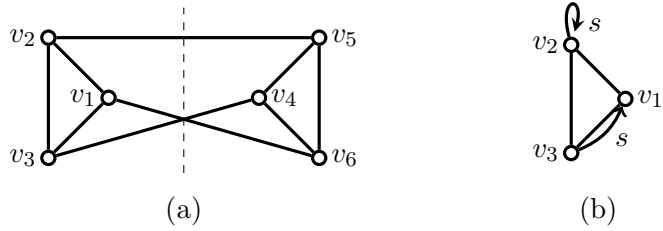

Conversely, let $(H,\psi)$ be a finite $\Gamma$-gain graph.
We simply denote a pair $(g,v)$, where $g\in \Gamma$ and $v\in V(H)$,  by $gv$.
The {\em covering graph} (also known as the derived graph) of $(H,\psi)$
is the simple graph with the vertex set $\Gamma\times V(H)=\{gv\mid g\in \Gamma, v\in V(H)\}$
and the edge set $\{\{gu,g\psi(e)v\}\mid e=(u,v)\in E(H), g\in \Gamma\}$.
%We denote it by $\Gamma G$.

Clearly, $\Gamma$  acts freely on the covering graph with the action
$\theta$ defined by $\theta(g):v \mapsto gv$ for $g\in \Gamma$,
under which the quotient graph comes back to $(H,\psi)$.
In this way, there is a one-to-one correspondence between $\Gamma$-gain graphs
and $\Gamma$-symmetric graphs with free actions (up to the choices of representative vertices).

The map $c:G\rightarrow H$ defined by $c(gv)=v$ and $c(\{gu,g\psi(e)v\})=(u,v)$ is called a {\em  covering map}.
In order to avoid confusion, throughout the paper, a vertex or an edge in a quotient gain graph $H$
is denoted with the mark tilde, e.g., $\tilde{v}$ or $\tilde{e}$.
Then the fiber $c^{-1}(\tilde{v})$ of a vertex $\tilde{v}\in V(H)$ and the fiber $c^{-1}(\tilde{e})$
of an edge $\tilde{e}\in E(H)$
coincide with a vertex orbit and an edge orbit, respectively, in $G$.

%%%%%%%%%%%%%%%%%%%%%%%%%%%%%%%%%%%%%%%%%%%%%%%%%%

\subsection{Symmetric bar-joint frameworks}
\label{subsec:symmetric_framework}

Given a finite simple graph $G$ and a map $\bp:V(G)\to \mathbb{R}^d$,
a \emph{symmetry operation} of the framework $(G,\bp)$ in $\mathbb{R}^{d}$
is an isometry $x$ of $\mathbb{R}^{d}$ such that for some $\alpha_x\in \textrm{Aut}(G)$, we have
\begin{equation} \label{eq:symop} x(p_i)=p_{\alpha_x(i)} \qquad \textrm{for all }
i\in V(G)\textrm{. }\nonumber\end{equation}
The set of all symmetry operations of a framework $(G,\bp)$ forms a group under composition, called the \emph{point group} of $(G,\bp)$.
Since translating a framework does not change its rigidity properties, we may assume wlog that the point group of a framework is always a \emph{symmetry group}, i.e., a subgroup of the orthogonal group $O(\mathbb{R}^{d})$.

Given a symmetry group $\Gamma$ and a graph $G$,
we let ${\cal R}_{(G,\Gamma)}$ denote the set of all $d$-dimensional realizations
of $G$ whose point group is either equal to $\Gamma$ or contains $\Gamma$ as a subgroup \cite{BS1,BS2,BS3,BS4}.
In other words, the set ${\cal R}_{(G,\Gamma)}$ consists of all realizations $(G,\bp)$ of $G$
for which there exists an action $\theta:\Gamma\to \textrm{Aut}(G)$ so that
\begin{equation}\label{class} x\big(\bp(v))=\bp({\theta(x)(v)}) \qquad \textrm{ for all } v\in V(G)\textrm{ and all } x\in \Gamma\textrm{.}\end{equation}
A framework $(G,\bp)\in {\cal R}_{(G,\Gamma)}$ satisfying the equations in (\ref{class})
for $\theta:\Gamma\to \textrm{Aut}(G)$ is said to be \emph{of type $\theta$},
and the set of all realizations in ${\cal R}_{(G,\Gamma)}$ which are of type $\theta$ is denoted
by ${\cal R}_{(G,\Gamma,\theta)}$ (see again \cite{BS1,BS2,BS3} and Figure \ref{K33types}).
It is shown in \cite{BS1} that
$(G,\bp)$ is of a unique type $\theta$ and $\theta$ is necessarily also a homomorphism,
when $\bp$ is injective.

\begin{figure}[htp]
\begin{center}
           \begin{tikzpicture}[very thick,scale=1]
\tikzstyle{every node}=[circle, draw=black, fill=white, inner sep=0pt, minimum width=5pt];
    \path (-0.7,0.8) node (p1) [label = above left: $p_1$] {} ;
    \path (0.7,0.8) node (p4) [label = above right: $p_4$] {} ;
    \path (-1.6,-0.5) node (p2) [label = below left: $p_2$] {} ;
     \path (1.6,-0.5) node (p3) [label = below right: $p_3$] {} ;
      \draw (p1) -- (p4);
    \draw (p1) -- (p2);
    \draw (p3) -- (p4);
    \draw (p2) -- (p3);
     \draw [dashed, thin] (0,-1) -- (0,1.3);
     \node [draw=white, fill=white] (z) at (-0.2,-1) {$s$};
      \node [draw=white, fill=white] (b) at (0,-1.6) {(a)};
        \end{tikzpicture}
        \hspace{2cm}
      \begin{tikzpicture}[very thick,scale=1]
\tikzstyle{every node}=[circle, draw=black, fill=white, inner sep=0pt, minimum width=5pt];
       \tikzstyle{every node}=[circle, draw=black, fill=white, inner sep=0pt, minimum width=5pt];
    \path (-0.7,0.8) node (p1) [label = above left: $p_1$] {} ;
    \path (0.7,0.8) node (p4) [label = above right: $p_3$] {} ;
    \path (-1.6,-0.5) node (p2) [label = below left: $p_2$] {} ;
     \path (1.6,-0.5) node (p3) [label = below right: $p_4$] {} ;
      \draw (p1) -- (p3);
    \draw (p1) -- (p2);
    \draw (p3) -- (p4);
    \draw (p2) -- (p4);
     \draw [dashed, thin] (0,-1) -- (0,1.3);
     \node [draw=white, fill=white] (z) at (-0.2,-1) {$s$};
      \node [draw=white, fill=white] (b) at (0,-1.6) {(b)};
        \end{tikzpicture}
\end{center}
\vspace{-0.3cm}
\caption{$2$-dimensional realizations of $K_{2,2}$ in ${\cal R}_{(K_{2,2},\mathcal{C}_s)}$ of different types: the framework in (a) is of type
$\theta_{a}$, where $\theta_{a}: \mathcal{C}_{s} \to \textrm{Aut}(K_{2,2})$ is the homomorphism defined by $\theta_{a}(s)=
(1 \, 4)(2 \, 3)$, and the framework in (b) is of type
$\theta_{b}$, where $\theta_{b}: \mathcal{C}_{s} \to \textrm{Aut}(K_{2,2})$ is the homomorphism defined by $\theta_{b}(s)=
(1 \, 3)(2\, 4)$.}
\label{K33types}
\end{figure}
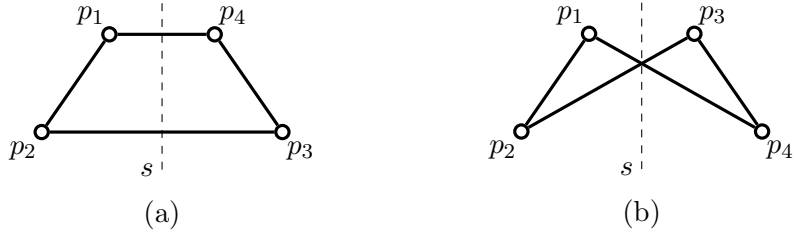

For simplicity, we will assume throughout this paper that a framework $(G,\bp)\in {\cal R}_{(G,\Gamma)}$ has no joint that is `fixed' by a non-trivial symmetry operation in $\Gamma$ (i.e., $(G,\bp)$ has no joint $p_i$ with $x(p_i)=p_i$ for some $x\in \Gamma$, $x\neq id$).

%(In the subsequent discussion it will be  convenient to use $\Gamma$ as an abstract group,
%rather than a specific point group to cope with, say ${\cal C}_s$ and ${\cal C}_2$ simultaniously.)
Let $\Gamma$ be an abstract group, and   $G$ be a  $\Gamma$-symmetric graph with respect to
a free action $\theta:\Gamma\rightarrow {\rm Aut}(G)$.
Suppose also that $\Gamma$ acts on $\mathbb{R}^d$
via a homomorphism $\tau:\Gamma\rightarrow O(\mathbb{R}^d)$.
Then we say that a framework $(G,\bp)$ is {\em $\Gamma$-symmetric} (with respect to $\theta$ and $\tau$) if
$(G,\bp)\in {\cal R}_{(G,\tau(\Gamma),\theta)}$, that is, if
\begin{equation}
\label{eq:symmetric_func}
\tau(\gamma) (\bp(v))=\bp(\theta(\gamma) v) \qquad \text{for all } \gamma\in \Gamma \text{ and all } v\in V(G).
\end{equation}
%The pair $(G, \bp)$ is said to be a {\em $\Gamma$-symmetric framework}
%if $G$ and $\bp$ are $\Gamma$-symmetric.

Let $H$ be the quotient graph of $G$ with the covering map $c:G\rightarrow H$.
It is convenient to fix a representative vertex $v$ of each vertex orbit $\Gamma v=\{gv\colon g\in \Gamma\}$,
and define the {\em quotient} of $\bp$ to be $\tilde{\bp}:V(H)\rightarrow \mathbb{R}^d$,
so that there is a one-to-one correspondence between $\bp$ and $\tilde{\bp}$ given by
$\bp(v)=\tilde{\bp}(c(v))$ for each representative vertex $v$.

%In the subsequent discussion, we shall only consider finite frameworks, where $G$ is a finite graph.
%(In \textsection\ref{sec:further}, we will discuss infinite frameworks with crystallographic symmetry.)
%Namely, we shall restrict our attention to {\em discrete point groups} ${\cal P}$,
%which are finite discrete subgroups of the {\em orthogonal group} ${\cal O}(\mathbb{R}^d)$,
%i.e., the set of $d\times d$ orthogonal matrices.

For a discrete point group $\Gamma$, let $\mathbb{Q}_{\Gamma}$ be the field
 generated by $\mathbb{Q}$ and the entries of the matrices in $\Gamma$.
We say that $\bp$ (or $\tilde{\bp}$) is {\em  $\Gamma$-generic}
if the set of coordinates of the image of $\tilde{\bp}$ is algebraically independent over $\mathbb{Q}_{\Gamma}$.
Note that this definition does not depend on the choice of representative vertices. A $\Gamma$-symmetric framework $(G,\bp)$ is called
\emph{$\Gamma$-generic} if $\bp$ is $\Gamma$-generic.

Further, we say that $(G,\bp)$ is \emph{$\Gamma$-regular} if the rigidity matrix $R(G,\bp)$ has
maximal rank  among all $\Gamma$-symmetric realizations of $G$ (see also \cite{BS1}). If a framework is $\Gamma$-generic, then it is clearly also $\Gamma$-regular.

%%%%%%%%%%%%%%%%%%%%%%%%%%%%%%%%%%%%%%%%%%%%%%%%%%%%%%%%%%%%%%%%%%%%%%%%%%%%%%%%%%%%%%%%%%%%%%%%%%%%%%%%%%%%%%%%%

\subsection{Block-diagonalization of the rigidity matrix}\label{sec:blockdia}
It is shown in \cite{KG2,BS2} that the rigidity matrix of a symmetric framework
 can be transformed into a block-diagonalized form using techniques from group representation theory.
In the following, we will briefly present the details of this fundamental result in order to clarify the
 combinatorics underlying our further analyses in the subsequent sections.

%To see this, we first introduce some terminologies.
For an $m\times n$ matrix $A$ and a $p\times q$ matrix $B$, $A\otimes B$ denotes
the {\em Kronecker product} of $A$ and $B$.
% is
%a $mp\times nq$-matrix of the form
%\begin{equation*}
%A\otimes B=
%\begin{pmatrix}
%a_{11}B & \dots & a_{1n}B \\
%\vdots &        & \vdots \\
%a_{m1} &  \dots       & a_{mn}B
%\end{pmatrix}.
%\end{equation*}
The following are well-known properties of this algebraic operation:
\begin{description}
\item $(A+B)\otimes C=A\otimes C+B\otimes C$  and $C\otimes (A+B)=C\otimes A+C\otimes B$.
\item $(A\otimes B)(C\otimes D)=(AC)\otimes (BD)$.
\item $(A\otimes B)^{\top}=A^{\top}\otimes B^{\top}$.
%\item $A\otimes B$ is orthgonal if $A$ and $B$ are orthgonal.
\end{description}
Given two matrix representations $\rho_1$ and $\rho_2$ of a group $\Gamma$,
the {\em tensor product} $\rho_1\otimes \rho_2$ is defined by
$\rho_1\otimes \rho_2(\gamma)=\rho_1(\gamma)\otimes \rho_2(\gamma)$ for $\gamma\in \Gamma$.
%By disctibutive law, $\rho\otimes \phi$ is again a matrix representation of $\Gamma$.

A matrix $M:\mathbb{R}^d\rightarrow \mathbb{R}^h$ is called a {\em $\Gamma$-linear map} of
$\rho_1$ and $\rho_2$  if $M\rho_1(\gamma)=\rho_2(\gamma)M$ for $\gamma \in \Gamma$.
The set of all $\Gamma$-linear maps of $\rho_1$ and $\rho_2$
forms a linear space which is denoted by ${\rm Hom}_{\Gamma}(\rho_1,\rho_2)$.

Let $(G,\bp)$ be a $\Gamma$-symmetric framework with respect
to a free action $\theta:\Gamma\rightarrow {\rm Aut}(G)$ and a homomorphism $\tau:\Gamma\rightarrow O(\mathbb{R}^d)$.
We denote by $P_V:\Gamma\rightarrow GL(\mathbb{R}^V)$
the linear representation of $\Gamma$ induced by $\theta$ over $V(G)$,
that is,
$P_V(\gamma)$ is the
permutation matrix of the permutation $\theta(\gamma)$ of $V(G)$.
Specifically, $P_V(\gamma)=[\delta_{i,\theta(\gamma)(j)})]_{i,j}$, where $\delta$ denotes the Kronecker delta symbol.
Similarly, let $P_E:\Gamma\rightarrow GL(\mathbb{R}^E)$
be the linear representation of $\Gamma$ consisting of
permutation matrices of permutations induced by $\theta$ over $E(G)$.
%Note that both $P_V$ and $P_E$ are orthogonal.

Let $\vec{G}$ be a directed graph obtained from $G$ by assigning an orientation to each edge
so that it preserves the action $\theta$
(i.e., an edge $\{u,v\}$ is directed from $u$ to $v$ if and only if
$\{\gamma u,\gamma v\}$ is directed from $\gamma u$ to $\gamma v$).
The incidence matrix $I_{\vec{G}}$ of $\vec{G}$ is the $|E(G)|\times |V(G)|$ matrix,
where the row of $e=(i,j)\in E(\vec{G})$ has the entries $-1$ and $1$ in the columns of $i$ and $j$, respectively,
and the other entries are zero.

It is important to notice that since $\theta$ is an action on $G$
we have $I_{\vec{G}}\in {\rm Hom}_{\Gamma}(P_V,P_E)$.
To see this, we let for each $e\in E(G)$, $I_e$ be the $|E(G)|\times |V(G)|$ matrix
obtained from $I_{\vec{G}}$ by changing each entry to zero except those in the row of $e$.
Then $I_{\vec{G}}=\sum_{e\in E(\vec{G})} I_e$, and we can easily verify that
\[
 P_E(\gamma)I_e P_V(\gamma)^{\top}=I_{\theta(\gamma)(e)} \qquad \text{for all } \gamma\in \Gamma.
\]
This relation can naturally be extended to rigidity matrices,
as shown in \cite{BS2,owen}. Here we give a short proof.
\begin{theorem}
\label{thm:block}
Let $\Gamma$ be a finite group with $\tau:\Gamma\rightarrow O(\mathbb{R}^d)$,
$G$ be a $\Gamma$-symmetric graph with a free action $\theta$ (on $V(G)$)
and $(G,\bp)$ be a $\Gamma$-symmetric framework with respect to $\theta$ and  $\tau$.
Then $R(G,\bp)\in {\rm Hom}_{\Gamma}(\tau \otimes P_V,P_E)$.
\end{theorem}
\begin{proof}
Let $R_e$ be the $|E(G)|\times d|V(G)|$ matrix obtained from $R(G,\bp)$ by changing each entry to zero
except those in the row of $e$.
As above, we consider the directed graph $\vec{G}$, and for each $e=(u,v)$, we let $\bp(e)=\bp(v)-\bp(u)$.
Note that $R(G,\bp)=\sum_{e\in E(\vec{G})} R_e=\sum_{e\in E(\vec{G})} \bp(e)^{\top}\otimes I_{e}$,
where $I_e$ is defined as above.
For each $e\in E(\vec{G})$ and $\gamma \in \Gamma$, we now have
\begin{align*}
P_E(\gamma)(\bp(e)^{\top}\otimes I_{e})(\tau(\gamma)\otimes P_V(\gamma))^{\top}
&=P_E(\gamma)(\bp(e)^{\top}\tau(\gamma)^{\top})\otimes (I_{e}P_V(\gamma)^{\top}) \\
&=(\tau(\gamma)\bp(e))^{\top}\otimes (P_E(\gamma)I_{e}P_V(\gamma)^{\top}) \\
&=\bp(\theta(\gamma)(e))^{\top}\otimes I_{\theta(\gamma)(e)} \\
&=R_{\theta(\gamma)(e)},
\end{align*}
where for the third equation we used the fact that $(G,\bp)$ is $\Gamma$-symmetric and hence
$\tau(\gamma)\bp(e)=\tau(\gamma)(\bp(u)-\bp(v))=\bp(\theta(\gamma)(u))-\bp(\theta(\gamma)(v))=\bp(\theta(\gamma)(e))$.
Therefore,
we obtain $P_E(\gamma) R(G,\bp) (\tau^{\top}(\gamma)\otimes P_V(\gamma))=\sum_{e\in E(\vec{G})} R_{\theta(\gamma)(e)}=R(G,\bp)$.
\end{proof}

Since $R(G,\bp)\in {\rm Hom}_{\Gamma}(\tau \otimes P_V,P_E)$,
there are non-singular matrices $S$ and $T$ such that
$T^{\top}R(G,\bp)S$ is block-diagonalized, by Schur's lemma.
If $\rho_0,\ldots, \rho_r$ are the irreducible representations of $\Gamma$,
then for an appropriate choice of symmetry-adapted coordinate systems, the rigidity matrix takes on the following block form
\begin{equation}
\label{rigblocks}
T^{\top}R(G,\bp)S:=\widetilde{R}(G,\bp)
=\left(\begin{array}{ccc}\widetilde{R}_{0}(G,\bp) & & \mathbf{0}\\ & \ddots & \\\mathbf{0} &  &
\widetilde{R}_{r}(G,\bp) \end{array}\right)\textrm{,}
\end{equation}
where the submatrix block $\widetilde{R}_i(G,\bp)$ corresponds to the irreducible representation $\rho_i$ of $\Gamma$.
%and it describes the relationship between external displacement vectors on the joints
%and resulting internal distortion vectors in the bars of $(G,\bp)$ that are symmetric with respect to $\rho_i$.
%More precisely, the numbers of columns and rows of $\widetilde{R}_i(G,\bp)$ are
%the respective dimensions of the space of displacement vectors on the joints of $(G,\bp)$
%which are symmetric with respect to $\rho_i$,
%and the space of distortion vectors in the bars of $(G,\bp)$ which are symmetric with respect to $\rho_i$;
%further,
The kernel of $\widetilde{R}_i(G,\bp)$ consists of all infinitesimal motions of $(G,\bp)$
which are symmetric with respect to $\rho_i$
(see \cite{BS2} for details).

%%%%%%%%%%%%%%%%%%%%%%

\subsection{Fully-symmetric motions and the orbit rigidity matrix}
Suppose that $\rho_0$ is the trivial irreducible representation of $\Gamma$, i.e.,
$\rho_0(\gamma)=1$ for all $\gamma\in \Gamma$.
%The submatrix block $\widetilde{R}_0(G,\bp)$ which corresponds to $\rho_0$ describes the relationship
%between external displacement vectors on the joints and resulting internal distortion vectors in the bars of $(G,\bp)$
%which exhibit the full symmetry of $\Gamma$.
The kernel of $\widetilde{R}_0(G,\bp)$ consists of all infinitesimal motions of $(G,\bp)$
which exhibit the full symmetry of $\Gamma$ (see also Fig.~\ref{fulsym}).
Specifically, an infinitesimal motion $\bmm:V(G)\rightarrow \mathbb{R}^d$ of $(G,\bp)$ is called
{\em fully $\Gamma$-symmetric} if
\begin{equation}
\bmm(\theta(\gamma)v)=\tau(\gamma)\bmm(v) \qquad \text{ for all $v\in V(G)$ and $\gamma\in\Gamma$}.
\end{equation}
We say that $(G,\bp)$ is {\em symmetry-forced (infinitesimally) rigid} if every
fully $\Gamma$-symmetric infinitesimal motion is trivial.

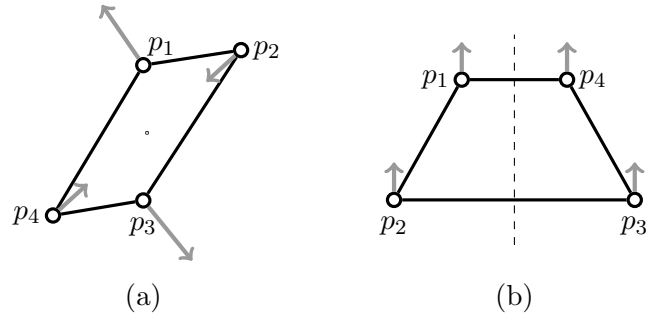
\begin{figure*}[htp]
\begin{center}
\begin{tikzpicture}[rotate=90,very thick,scale=1]
\tikzstyle{every node}=[circle, draw=black, fill=white, inner sep=0pt, minimum width=5pt];
    \path (0.1,1.2) node (p1) [label = left: $p_{4}$] {} ;
    \path (2.1,0) node (p4) [label = above right: $p_{1}$]{} ;
    \path (2.3,-1.3) node (p3) [label = right: $p_{2}$] {} ;
     \path (0.3,0) node (p2) [label = below: $p_{3}$] {} ;
        \draw (p1) -- (p4);
      \draw (p3) -- (p4);
     \draw (p2) -- (p3);
      \draw (p2) -- (p1);
            \draw [ultra thick, ->, black!40!white](p1) -- (0.52,0.74);
      \draw [ultra thick, ->, black!40!white](p3) -- (1.88,-0.84);
      \draw [ultra thick, ->, black!40!white](p2) -- (-0.5,-0.65);
      \draw [ultra thick, ->, black!40!white](p4) -- (2.9,0.55);
      \filldraw[fill=black, draw=black]
    (1.2,-0.05) circle (0.004cm);
     \node [draw=white, fill=white] (b) at (-1,-0) {(a)};
              \end{tikzpicture}
    \hspace{1cm}
            \begin{tikzpicture}[very thick,scale=1]
\tikzstyle{every node}=[circle, draw=black, fill=white, inner sep=0pt, minimum width=5pt];
    \path (-0.7,0.8) node (p1) [label = left: $p_{1}$] {} ;
    \path (0.7,0.8) node (p4) [label = right: $p_{4}$]{} ;
    \path (-1.6,-0.8) node (p2) [label = below: $p_{2}$] {} ;
     \path (1.6,-0.8) node (p3) [label = below: $p_{3}$] {} ;
      \draw (p1) -- (p4);
    \draw (p1) -- (p2);
    \draw (p3) -- (p4);
    \draw (p2) -- (p3);
     \draw [dashed, thin] (0,-1.4) -- (0,1.4);
     \draw [ultra thick, ->, black!40!white] (p1) -- (-0.7,1.3);
      \draw [ultra thick, ->, black!40!white] (p4) -- (0.7,1.3);
      \draw [ultra thick, ->, black!40!white] (p2) -- (-1.6,-0.3);
      \draw [ultra thick, ->, black!40!white] (p3) -- (1.6,-0.3);
      \node [draw=white, fill=white] (b) at (0,-2.1) {(b)};
        \end{tikzpicture}
 \end{center}
\vspace{-0.3cm}
\caption{Fully-symmetric infinitesimal motions of frameworks in the plane: (a) a  $\mathcal{C}_2$-symmetric non-trivial infinitesimal motion; (b) a $\mathcal{C}_s$-symmetric trivial infinitesimal motion.}
\label{fulsym}
\end{figure*}

To simplify the detection of fully $\Gamma$-symmetric motions of $(G,\bp)$,
the orbit rigidity matrix  of  $(G,\bp)$ was introduced in \cite{BSWW}.
The orbit rigidity matrix is equivalent to $\widetilde{R}_0(G,\bp)$, and has successfully been used for characterizing
symmetry-forced rigid frameworks in \cite{jkt,tan,MT1}.
In the next section, we will extend this concept to each irreducible representation of $\Gamma$.

%%%%%%%%%%%%%%%%%%%%%%%%%%%%%%%%%%%%%%%%%%%%%%%%%%%%%%%%%%%%%%%%%%%%%%%%%%%%%%%%%%%%%%%%%%%%%%%%%%%%%%%%%%%%%%%%%%%%%%%%%%%%%%%%%%%%%%%

\section{`Anti-symmetric' orbit rigidity matrices for bar-joint frameworks with Abelian point group symmetry}
\label{sec:block}
Let $(G,\bp)$ be a $\Gamma$-symmetric framework in $\mathbb{R}^d$
with respect to $\theta:\Gamma\rightarrow {\rm Aut}(G)$ and  $\tau:\Gamma\rightarrow O(\mathbb{R}^d)$.
In general, the entries of each block $\widetilde{R}_j(G,p)$ are not as simple as those of $\widetilde{R}_0(G,\bp)$.
However, if we restrict our attention to the case where $\Gamma$ is an Abelian group, then
we can specifically describe an `anti-symmetric' orbit rigidity matrix
for each of the irreducible representations of $\Gamma$.

For simplicity, we will first consider the case where $\Gamma$ is cyclic (Section~\ref{sec:cyclic}).
The argument is then easily extended to general Abelian groups in Section~\ref{subsec:non_cyclic}.
Throughout these two subsections we assume, again for the sake of simplicity, that $\theta$  acts freely on $E(G)$.
In Section~\ref{sec:nonfree}, we will discuss the case when $\theta$ may not be free  on $E(G)$.
In Section~\ref{sec:bar_join_texamples}, we give several examples.

\subsection{Case of cyclic groups}
\label{sec:cyclic}
Throughout this subsection,
 $\Gamma$ is assumed to be a cyclic group $\mathbb{Z}/k\mathbb{Z}=\{0,1,2,\dots, k-1\}$ of order $k$,
and $\theta$ acts freely on $E(G)$.
It is an elementary fact from group representation theory
that $\Gamma=\mathbb{Z}/k\mathbb{Z}$ has $k$ non-equivalent
irreducible representations  $\rho_0,\rho_1,\dots, \rho_{k-1}$,
and that each of these representations is one-dimensional.
Specifically, for $j=0,1,\dots, k-1$, we have
\begin{align*}
\rho_j:\Gamma&\rightarrow \mathbb{C}\setminus\{0\} \\
i &\mapsto \omega^{ij},
\end{align*}
where  $\omega$ denotes $e^{\frac{2\pi \sqrt{-1}}{k}}$, a root of unity.
To cope with such representations,  we need to extend the underlying field to $\mathbb{C}$ if $k\geq 3$,  and regard
$R(G,\bp)$ as a linear function from $(\mathbb{C}^k)^{V(G)}$
to $\mathbb{C}^{E(G)}$.
Next we show how each block $\widetilde{R}_j(G,\bp)$ is described in the complex field.

%%%%%%%%%%%%%%%%%%%%%%%%%%%%%%%%%%%%%%%%%%%%%%%%%%%%%%%%%%%%%%%%%%%%%%%%%%%%%%%%%%%%%%%%%%%%%%%%%%%%%%%%%%%%%%%%%%%%%%%%%%%%%%%

\subsubsection{Decompositions of the regular representation of $\Gamma$}

Let $\rho_{\rm reg}:\Gamma\rightarrow GL(\mathbb{R}^{k})$ be the regular representation of $\Gamma$,
that is, regarding $\Gamma$ as a subgroup of the symmetric group $S_{k}$,
$\rho_{\rm reg}(\gamma)=[\delta_{i,\gamma+ j}]_{i,j}$ for any $\gamma\in \Gamma$.
Recall that $\rho_{\rm reg}$ is equivalent to $\bigoplus_{j=0}^{k-1} \rho_j$.
%To describe the submatrix blocks $\widetilde{R}_j(G,\bp)$, we will make use of this decomposition of the regular representation.

For $j=0,1,\dots, k-1$, let
$b_j=(1,\bar{\omega}^{j},\bar{\omega}^{2j},\dots, \bar{\omega}^{(k-1)j})^{\top}$ be a vector in
$\mathbb{C}^k$, where $\bar{\omega}$ is the complex conjugate of $\omega$.
Then we have
\[
 \rho_{\rm reg}(i)b_j=\omega^{ij}b_j=\rho_j(i)b_j.
\]
This says that $b_j$ is a common eigenvector of $\{\rho_{\rm reg}(i)\mid i=0,1,\dots, k-1\}$,
and the one-dimensional subspace $I_j$ spanned by $b_j$ is an invariant subspace
corresponding to $\rho_j$.
Hence, by decomposing $\mathbb{C}^k$ into $\bigoplus_{j=0}^{k-1} I_j$,
$\rho_{\rm reg}$ is diagonalized  to $\bigoplus_{j=0}^{k-1} \rho_j$.

Next, consider $\tau\otimes \rho_{\rm reg}$.
Since the character of the Kronecker product of two representations is written by the coordinate-wise product of the
corresponding two characters, we see that the multiplicity of $\rho_j$ in $\tau\otimes \rho_{\rm reg}$ is equal to
${\rm Trace}(\tau(0))$, that is, equal to $d$. Hence, $\tau\otimes \rho_{\rm reg}$ is equivalent to
$\bigoplus_{j=0}^{k-1} d\rho_j$.

For $j=0,1,\dots, k-1$, we define a $d$-dimensional subspace $J_j$ of $\mathbb{C}^{dk}$ by
\begin{equation}
\label{eq:i_sym_mot}
 J_j=\left\{\begin{bmatrix} \tau(0) \\ \bar{\omega}^{j}\tau(1) \\ \vdots \\ \bar{\omega}^{j(k-1)}\tau(k-1)\end{bmatrix}x\colon x\in \mathbb{C}^{d}
 \right\}
\end{equation}
(where $\begin{bmatrix} \tau(0) \\ \vdots \\\bar{\omega}^{j(k-1)}\tau(k-1)\end{bmatrix}$ denotes a $dk\times d$ matrix).
Then observe that for each $i\in \Gamma$,
\[
 (\tau \otimes \rho_{\rm reg}(i)) y =\rho_j(i) y \qquad \text{for all } y \in J_j,
\]
and hence $J_j$ is a common eigenspace of $\{\tau\otimes \rho_{\rm reg}(i)\colon i=0,\dots, k-1\}$,
and $J_j$ is an invariant subspace corresponding to $\rho_j$.
$\mathbb{C}^{dk}$ is thus decomposed into invariant subspaces $\bigoplus_{j=0}^{k-1} J_j$.

%%%%%%%%%%%%%%%%%%%%%%%%%%%%%%%%%%%%%%%%%%%%%%%%%%%%%%%%%%%%%%%%%%%%%%%%%%%%%%%%%%%%%%%%%%%%%%%%%%%%%%%%%%%%%%%%%%%%

\subsubsection{Decompositions of $P_E$ and $\tau\otimes P_V$}

Since our goal is to characterize the infinitesimal rigidity of symmetric frameworks in terms of their quotient graphs,
let us introduce a quotient $\Gamma$-gain graph $(H,\psi)$ of $G$ with a covering map $c:G\rightarrow H$.

Observe, then, that since $\Gamma$ acts freely  on $V(G)$,
%(i.e., $\theta(\gamma)(v)\neq v$ for all $v\in V(G)$ and all $\gamma\in \Gamma$, $\gamma\neq id$),
$P_V$ is the direct sum of $|V(H)|$ copies of $\rho_{\rm reg}$, each of which represents an action of
$\Gamma$ over a fiber $c^{-1}(v)$.
Thus, $P_V=\bigoplus_{\tilde{v}\in V(H)} \rho_{\rm reg}$, and $P_V$ is equivalent to
$\bigoplus_{j=0}^{k-1} |V(H)|\rho_j$.
Similarly, if we assume that $\Gamma$ acts freely on $E(G)$,
then $P_E=\bigoplus_{\tilde{e}\in E(H)} \rho_{\rm reg}$,
and $P_E$ is equivalent to $\bigoplus_{j=0}^{k-1} |E(H)|\rho_j$.
(We will treat the case where $\Gamma$ does not act freely  on the edge set of $G$ in Section~\ref{sec:nonfree}.)

Observe also that
$\tau\otimes P_V=\tau\otimes (\bigoplus_{\tilde{v}\in V(H)} \rho_{\rm reg})=\bigoplus_{\tilde{v}\in V(H)}\tau\otimes \rho_{\rm reg}$.
Thus, $\tau\otimes P_V$ is equivalent to $\bigoplus_{j=0}^{k-1} d|V(H)| \rho_j$.
In total, each block
$\widetilde{R}_j(G,p)$ corresponding to $\rho_j$ has the size $|E(H)|\times d|V(H)|$.

The decompositions of $P_E$ and $\tau\otimes P_V$ give us further information about $\widetilde{R}_j(G,\bp)$.
Since $\Gamma$ acts freely  on $G$,
each vertex orbit is associated with a $dk$-dimensional subspace of $(\mathbb{C}^d)^{V(G)}$, while
each edge orbit is associated with a $k$-dimensional subspace of $\mathbb{C}^{E(G)}$.
In other words,
$\mathbb{C}^{V(G)}$
and $\mathbb{C}^{E(G)}$  can be written
as $\bigoplus_{\tilde{v}\in V(H)} \mathbb{C}^{dk}$ and  $\bigoplus_{\tilde{e}\in E(H)} \mathbb{C}^k$ in terms of the quotient graph $H$.

Since $\tau\otimes P_V=\bigoplus_{\tilde{v}\in V(H)} \tau \otimes \rho_{\rm reg}$ and
 $P_E=\bigoplus_{\tilde{e}\in E(H)} \rho_{\rm reg}$,
it follows that
$J_j^{\rm mo}:=\bigoplus_{\tilde{v}\in V(H)} J_j$ is an invariant subspace of $\mathbb{C}^{V(G)}$ while
$I_j^{\rm st}:=\bigoplus_{\tilde{e}\in E(H)} I_j$ is an invariant subspace of $\mathbb{C}^{E(G)}$ with respect to $\rho_j$.
Therefore $\widetilde{R}_j(G,\bp)$ is a linear mapping from $J_j^{\rm mo}$ to $I_j^{\rm st}$.

An infinitesimal motion $\bmm:V(G)\rightarrow \mathbb{C}^d$ contained in $J_j^{\rm mo}$ is
said to be {\em $\rho_j$-symmetric}.
By definition (\ref{eq:i_sym_mot}), $\bmm$ is $\rho_j$-symmetric if and only if
\begin{equation}
\label{eq:j_symmetric}
 \bmm(\gamma v)=\bar{\omega}^{j\gamma}\tau(\gamma) \bmm(v) \qquad \text{for all $\gamma\in \Gamma$ and $v\in V(G)$}.
\end{equation}

Recall that $\bmm:V(G)\rightarrow \mathbb{C}^d$ is an infinitesimal motion of $(G,\bp)$ if
\begin{equation}
\label{eq:inf}
 \langle \bp(u)-\bp(v), \bmm(u)-\bmm(v)\rangle=0 \qquad \text{for all } \{u,v\}\in E(G).
\end{equation}
This system of linear equations for $\bmm$ is redundant if $\bmm$ is restricted to be $\rho_j$-symmetric,
and we now eliminate such redundancy  as follows.

Recall that each edge orbit is written as
a set $c^{-1}(\te)=\{\{\gamma u,\gamma \psi_{\te} v\}\colon \gamma\in \Gamma\}$
of edges of $G$,
where $\psi_{\te}$ is the label assigned to $\te$ in $(H,\psi)$.
So (\ref{eq:inf}) can be written as
\begin{equation}
\langle \bp(\gamma u)-\bp(\gamma \psi_{\te} v), \bmm(\gamma u)-\bmm(\gamma \psi_{\te} v)\rangle=0 \qquad (\gamma\in \Gamma)
\end{equation}
for each $\te\in E(H)$.
By the symmetry of $\bp$ and $\bmm$ with respect to $\Gamma$, these $k$ equations can be simplified to one equation
\begin{equation}
\label{eq:j_inf}
\langle \bp(u)-\tau(\psi_{\te})\bp(v), \bmm(u)-\bar{\omega}^{j\psi_{\te}}\tau(\psi_{\te}) \bmm(v)\rangle=0
\end{equation}
for each edge orbit.

Let us define the joint $\tilde{\bp}(\tilde{w})$ and the motion $\tilde{\bmm}(\tilde{w})$ of a vertex $\tilde{w}\in V(H)$
to be the joint $\bp(v)$ and the motion $\bmm(v)$ of the representative vertex $v$ of the vertex orbit $c^{-1}(\tilde{w})$.
Then the analysis can be done on the quotient graph $(H,\psi)$.
More formally, for a $\Gamma$-gain graph $(H,\psi)$ and $\tilde{\bp}:V(H)\rightarrow \mathbb{R}^d$,
a map $\tilde{\bmm}:V(H)\rightarrow \mathbb{C}^d$ is said to be a $\rho_j$-symmetric motion of $(H,\psi,\tilde{\bp})$ if
\begin{equation}
\label{eq:j_inf_quot}
\langle \tilde{\bp}(\tilde{u})-\tau(\psi_{\te})\tilde{\bp}(\tilde{v}), \tilde{\bmm}(\tilde{u})-\bar{\omega}^{j\psi_{\te}}\tau(\psi_{\te}) \tilde{\bmm}(\tilde{v})\rangle=0 \qquad \text{for all } \te=(\tilde{u},\tilde{v})\in E(H).
\end{equation}
We define the $\rho_j$-orbit rigidity matrix, denoted by $O_j(H,\psi,\tilde{\bp})$,  as the  $|E(H)|\times d|V(H)|$ matrix
associated with the system (\ref{eq:j_inf_quot}),
where each vertex has the corresponding $d$ columns,
each edge has the corresponding row, and the row corresponding to
$\te=(\tilde{u},\tilde{v})\in E(H)$ is given by
\[\begin{array}{ccccc}
       & \overbrace{\hspace{25mm}}^{\tilde{u}} & & \overbrace{\hspace{40mm}}^{\tilde{v}} & \\
      0\dots0 & \tilde{\bp}(\tilde{u})-\tau(\psi_{\te})\tilde{\bp}(\tilde{v}) & 0\dots0 & \omega^{j\psi_{\te}}(\tilde{\bp}(\tilde{v})-\tau(\psi_{\te})^{-1}\tilde{\bp}(\tilde{u})) & 0\dots0 \\
    \end{array},\]
where each vector is assumed to be transposed, and if $\te$ is a loop at $\tilde{v}$ the entries of $\tilde{v}$ become the sum of the two entries given above.

%\[\begin{array}{ccc}
%       & \overbrace{\hspace{40mm}}^{v}  & \\
%      0\dots0 & (2I_d-\phi(e)-\phi(e)^{-1})\tilde{p}(v) & 0\dots0 \\
%    \end{array}\]
% if $e$ is a loop.
Due to the one-to-one correspondence between $J_j^{\rm mo}$ and $(\mathbb{C}^d)^{V(H)}$, we conclude the
following.
\begin{proposition}
\label{prop:cyclic}
Let $\Gamma$ be a cyclic group of order $k$, $(G,\bp)$ be a $\Gamma$-symmetric framework in $\mathbb{R}^d$,
and $(H,\psi)$ be the quotient $\Gamma$-gain graph.
Then, for each $j=0,\dots, k-1$
\[
 {\rm rank}\ \widetilde{R}_j(G,\bp)={\rm rank}\ O_j(H,\psi,\tilde{\bp}).
\]
\end{proposition}

%%%%%%%%%%%%%%%%%%%%%%%%%%%%%%%%%%%%%%%%%%%%%%%%%%%%%%%%%%%%%%%%%%%%%%%%%%%%%%%%%%%%%%%%%%%%%%%%%%%%%%%%%%%%%%%%%%%%%%%%%%%%%%%%%%%%

\subsection{Case of non-cyclic groups}
\label{subsec:non_cyclic}
It is well known that any finite Abelian group $\Gamma$ is
isomorphic to $\mathbb{Z}/k_1\mathbb{Z}\times \dots \times \mathbb{Z}/k_l\mathbb{Z}$
for some positive integers $k_1,\dots, k_l$.
Thus, we may denote each element of $\Gamma$ by ${\bm i}=(i_1,\dots,i_l)$, where
 $0\leq i_1\leq k_1,\dots, 0\leq i_l\leq k_l$, and regard $\Gamma$ as an additive group.

Let $k=|\Gamma|=k_1 k_2 \dots k_l$.
$\Gamma$ has $k$ non-equivalent irreducible representations
which are denoted by  $\{\rho_{\bm j}\colon {\bm j}\in \Gamma\}$.
Specifically, for each ${\bm j}\in \Gamma$, $\rho_{\bm j}$ is defined by
\begin{align} \label{eq:abelian_rho}
\rho_{{\bm j}}:\Gamma&\rightarrow \mathbb{C}/\{0\}  \nonumber\\
{\bm i}&\mapsto \omega_1^{i_1j_1}\cdot\omega_2^{i_2j_2} \cdot \ldots \cdot\omega_l^{i_lj_l} ,
\end{align}
where $\omega_t=e^{\frac{2\pi \sqrt{-1}}{k_t}}$, $t=1,\ldots, l$.

%\begin{align*}
%\rho_{{\bm j}}:\Gamma&\rightarrow \mathbb{C}/\{0\}  \\
%{\bm i}&\mapsto \omega^{\langle {\bm i}, {\bm j}\rangle},
%\end{align*}
%where $\langle {\bm i}, {\bm j}\rangle$ means $i_1j_1+\dots+i_lj_l$.

We now apply the analysis for cyclic groups by simply replacing each index with a tuple of indices.
By Theorem~\ref{thm:block}, $R(G,\bp)$ is decomposed into $k$ blocks,
and the block corresponding to $\rho_{\bm j}$ is denoted by $\widetilde{R}_{\bm j}(G,\bp)$.

For each ${\bm j}=(j_1,\dots, j_l)\in \Gamma$,
let $b_{\bm j}$ be the $k$-dimensional vector such that
each coordinate is indexed by a tuple ${\bm i}\in \Gamma$ and
its ${\bm i}$-th coordinate is equal to $\bar{\omega}_1^{i_1j_1} \cdot \ldots \cdot\bar{\omega}_l^{i_lj_l}$.
Then, for  the regular representation $\rho_{\rm reg}$ of $\Gamma$, we have
\[
 \rho_{\rm reg}({\bm i})b_{{\bm j}}=\omega_1^{i_1j_1}\cdot \ldots \cdot\omega_l^{i_lj_l} b_{\bm j}
=\rho_{\bm j}({\bm i})b_{\bm j},
\]
and hence $b_{\bm j}$ is a common eigenvector of $\{\rho_{\rm reg}({\bm i})\mid {\bm i}\in \Gamma\}$.
Hence, the one-dimensional subspace $I_{\bm j}$ spanned by $b_{\bm j}$ is an invariant subspace of
$\mathbb{C}^{k}$ corresponding to $\rho_{\bm j}$.

A similar analysis determines the common eigenspace $J_{\bm j}$  of
$\{\tau\otimes \rho_{\rm reg}({\bm i})\mid {\bm i}\in \Gamma\}$ for the eigenvalue $\rho_{\bm j}({\bm i})$ as a counterpart to
the one defined in (\ref{eq:i_sym_mot}).

Following the analysis given in the previous subsection, we see that
$\widetilde{R}_{\bm j}(G,\bp)$ is a linear mapping from $J_{\bm j}^{\rm mo}:=\bigoplus_{\tv\in V(H)}J_{\bm j}$
to $I_{\bm j}^{\rm st}:=\bigoplus_{\te\in E(H)} I_{\bm j}$.
If we define the $\rho_{\bm j}$-orbit rigidity matrix, denoted by $O_{\bm j}(H,\psi,\tilde{\bp})$,
as the $|E(H)|\times d|V(H)|$ matrix, where each $\te=(\tilde{u},\tilde{v})\in E(H)$ has the associated row
\[\begin{array}{ccccc}
       & \overbrace{\hspace{25mm}}^{\tilde{u}} & & \overbrace{\hspace{45mm}}^{\tilde{v}} & \\
      0\dots0 & \tilde{\bp}(\tilde{u})-\tau(\psi_{\te})\tilde{\bp}(\tilde{v}) & 0\dots0 & \rho_{\bm j}(\psi_e)(\tilde{\bp}(\tilde{v})-\tau(\psi_{\te})^{-1}\tilde{\bp}(\tilde{u})) & 0\dots0 \\
    \end{array},\]
then we have the following result.
\begin{proposition}
\label{prop:abelian}
Let $\Gamma$ be a finite Abelian group, $(G,\bp)$ be a $\Gamma$-symmetric framework in $\mathbb{R}^d$,
and $(H,\psi)$ be the quotient $\Gamma$-gain graph.
Then, for each ${\bm j}\in \Gamma$,
\[
 {\rm rank}\ \widetilde{R}_{\bm j}(G,\bp)={\rm rank}\ O_{\bm j}(H,\psi,\tilde{\bp}).
\]
\end{proposition}

%%%%%%%%%%%%%%%%%%%%%%%%%%%%%%%%%%%%%%%%%%%%%%%%%%%%%%%%%%%%%%%%%%%%%%%%%%%%%%%%%%%%%%%%%%%%%%%%%%%%%%%%%%%%%%%%%%%%%%%%%%%%%%%%%%%%%%%

\subsection{Group actions which are not free on the edge set}\label{sec:nonfree}

In the previous sections, we restricted ourselves to the situation, where the group $\Gamma$ acts freely on both the vertex set and the edge set of the graph $G$. Let us now also consider the case, where $\Gamma$ acts freely on the vertex set, but not on the edge set of $G$. In other words, there exists an element $\gamma\in\Gamma$ with $\theta(\gamma)(u)= v$ and $\theta(\gamma)(v)= u$ for some $\{u,v\}\in E(G)$.
Since $\Gamma$ still acts freely on $V(G)$, it follows that if $\Gamma$ does not act freely on $c^{-1}((\tilde{u},\tilde{v}))$, then the edge orbit of $(\tilde{u},\tilde{v})$ is of size $\frac{|\Gamma|}{2}$, that is, $\Gamma/(\mathbb{Z}/2\mathbb{Z})$ acts freely on $c^{-1}((\tilde{u},\tilde{v}))$.

Now, let $(G,\bp)$ be a $\Gamma$-symmetric framework, where $\Gamma$ is a finite Abelian group of order $k$, and suppose there are $n$ edge orbits of size $k$ and $m$ edge orbits of size $\frac{k}{2}$.
Let $g_1,\ldots, g_t$ be the non-trivial elements of $\Gamma$ which fix an edge of $G$, and let $m_{i}$ be the number of edge orbits whose representatives are fixed by $g_i$. (Note that if an edge $e$ of $G$ is fixed by an element of $\Gamma$, then so is every other element in the orbit of $e$, because $\Gamma$ is Abelian.)
So we have $m=\sum_{i=1}^{t}m_{i}$, and the character of $P_E$ is the vector $\chi(P_E)$ which has $nk+m\frac{k}{2}$ in the first entry corresponding to $id\in \Gamma$, $m_{i}\frac{k}{2}$ in the entry corresponding to $g_i$, $i=1,\ldots, t$, and $0$ elsewhere.

 Now, let $\rho_{\bm j}$ be an irreducible representation of $\Gamma$.
Then,  since each $g_i$ must be an involution,
$\rho_{\bm j}(g_i)$ is $1$ or $-1$.
Without loss of generality assume $\rho_{\bm j}(g_i)=1$ for $1\leq i\leq s$
and $\rho_{\bm j}(g_i)=-1$ for $s+1\leq i\leq t$.
It is a well known result from group representation theory
that the dimension of the invariant subspace $I_{\bm j}^{st}$ of $\mathbb{C}^{|E(G)|}$ is given by $\frac{1}{k}(\chi(P_E)\cdot \rho_{\bm j})$.
Thus, \begin{eqnarray*}\textrm{dim} (I_{\bm j}^{st}) & = & \frac{1}{k}\Big(nk+m\frac{k}{2}+\sum_{i=1}^s m_{i}\frac{k}{2}- \sum_{i=s+1}^t m_{i}\frac{k}{2} \Big)\\
 & = & \frac{1}{k}(nk+ \sum_{i=1}^s m_{i}k)\\
 & = & n+ \sum_{i=1}^s m_{i}.
 \end{eqnarray*}
It follows that the submatrix block $\widetilde{R}_{\bm j}(G,\bp)$ has $n+ \sum_{i=1}^s m_{i}$ many rows.

%Note that in this case, $\Gamma/(\mathbb{Z}/2\mathbb{Z})$ acts freely on the edge orbit of $\{u,v\}$.Recall that if $\Gamma$ acts freely on $E(G)$, then the stress space is decomposed as $\bigoplus_{e\in E(H)} \rho_{reg}$, where $\rho_{reg}$ is the regular representation of $\Gamma$. Now, for each $e\in E(H)$ with the property that $\Gamma$ does not act freely on $c^{-1}(e)$, we need to  replace $\rho_{reg}$ (associated with $e$) by $\rho_{reg}^*$, where $\rho_{reg}^*$ is the regular representation of  $\Gamma/(\mathbb{Z}/2\mathbb{Z})$.

Although the size of $\widetilde{R}_{\bm_j}(G,\bp)$ and that of $O_{\bm j}(H,\psi,\tilde{\bp})$ are different,
we can still use $O_{\bm j}(H,\psi,\tilde{\bp})$ to compute the rank of $\widetilde{R}_{\bm j}(G,\bp)$,
as Proposition~\ref{prop:abelian} still holds.
To see this, observe that if $g_i$ fixes $c^{-1}(\te)$ for some $\te\in E(H)$,
 then $\te$ is a loop with $\psi(\te)=g_i$.
Since $g_i^2=id$, if $\rho_{\bm j}(g_i)=-1$,
the row corresponding to $\te$ in $O_{\bm j}(H,\psi,\tilde{\bp})$ turns out to be a zero vector.
The following proposition implies that the reverse implication is also true,
where a loop $\te$ is called a {\em zero loop in $O_{\bm j}(H,\psi,\tilde{\bp})$} if the row of $\te$
is a zero vector in $O_{\bm j}(H,\psi,\tilde{\bp})$.
\begin{proposition}
\label{lem:zero_loop}
Let $\Gamma$ be an Abelian  group along with a faithful representation
$\tau:\Gamma\rightarrow O(\mathbb{R}^d)$,
$(G,\bp)$ be a $\Gamma$-symmetric framework with respect to $\theta$ and $\tau$, and
$(H,\psi)$ be a quotient $\Gamma$-gain graph.
Then, for each ${\bm j}\in \Gamma$, a loop $\te$ is a zero loop  in $O_{\bm j}(H,\psi,\tilde{\bp})$
if and only if $\rho_{\bm j}(\psi_{\te})=-1$ and $\psi_{\te}^2=id$.
\end{proposition}
\begin{proof}
For simplicity, let $\omega=\rho_{\bm j}(\psi_{\te})$ and $A=\tau(\psi_{\te})\neq I_d$.
By definition, the row of $\te$ is a zero vector if and only if
$I_d+\omega I_d-A-\omega A^{-1}=0$.
The latter condition is equivalent to $(A-I_d)(A-\omega I_d)=0$.
This holds if $\omega=-1$ and $A^2=I_d$, which implies the sufficiency.

%Suppose $A\neq I_d$.
To see the necessity, let $\mu_A$ be the minimal polynomial of $A$.
%Then,
Since $A$ is diagonalizable (as $\Gamma$ is Abelian) and $\mu_A$ divides $(t-1)(t-\omega )$,
an elementary theorem of linear algebra implies that
the eigenvalues of $A$ are only $1$ and $\omega$.
Since $\Gamma$ is Abelian and $A\neq I_d$,  we have $\omega=-1$.
This also implies $A^2=I_d$, and hence $\psi_{\te}^2=id$.
\end{proof}

It follows from Proposition \ref{lem:zero_loop} and the remarks above that the number of rows of  $\widetilde{R}_{\bm j}(G,\bp)$ equals the number of non-zero rows of $O_{\bm j}(H,\psi,\tilde{p})$. Moreover, these two matrices clearly have the same number of columns, and by the same reasoning as in the previous sections, Propositions~\ref{prop:cyclic} and \ref{prop:abelian} still hold.

%%%%%%%%%%%%%%%%%%%%%%%%%%%%%%%%%%%%%%%%%%%%%%%%%%%%%%%%%%%%%%%%%%%%%%%%%%%%%%%%%%%%%%%%%%%%%%%%%%%%%%%%%%%%%%%%%%%%%%%%%%%%%%%%%%%

\subsection{Examples}\label{sec:bar_join_texamples}

\subsubsection{Reflection symmetry $\mathcal{C}_s$}

The symmetry group $\mathcal{C}_{s}$ has two non-equivalent real irreducible representations each of which is of dimension 1. In the Mulliken notation, they are denoted by $A'$ and $A''$ (see Table \ref{tabcs}).

\begin{table}[htp]
\begin{center}
\begin{tabular}{r||r|r}
$\mathcal{C}_{s}$   &   $id$  &  $s$   \\\hline\hline
$A'=\rho_0$  &    1  &  1 \\\hline
$A''=\rho_1$  &    1  &  -1 \\
\end{tabular}
\end{center}
\caption{The irreducible representations of $\mathcal{C}_s$.}
\label{tabcs}
\end{table}

It follows that the block-decomposed rigidity matrix $\widetilde{R}(G,\bp)$ of a $\mathcal{C}_{s}$-symmetric framework $(G,\bp)$ consists of only two blocks: the submatrix block $\widetilde{R}_0(G,\bp)$  corresponding to the trivial representation $\rho_0$,
and the submatrix block $\widetilde{R}_1(G,\bp)$ corresponding to the representation $\rho_1$. The block $\widetilde{R}_0(G,\bp)$ is equivalent to the (fully symmetric) orbit rigidity matrix (see also \cite{BSWW}). The block $\widetilde{R}_1(G,\bp)$ describes the $\rho_1$-symmetric (or simply `anti-symmetric') infinitesimal rigidity properties of $(G,\bp)$, where
an infinitesimal motion $\bmm$  of  $(G,\bp)$ is anti-symmetric if
 \begin{equation*} \tau(s)\big(\bmm_i\big)=-\bmm_{\theta(s)(i)}\textrm{ for all } i\in V(G)\textrm{,}\end{equation*} i.e., if all the velocity vectors of $\bmm$ are reversed by $s$ (see also Fig.~\ref{antisym}). As shown in Proposition~\ref{prop:cyclic}, $\widetilde{R}_1(G,\bp)$  is equivalent to the anti-symmetric orbit rigidity matrix  $O_1(H,\psi,\tilde{\bp})$.

\begin{figure}[htp]
\begin{center}
\begin{tikzpicture}[very thick,scale=1]
\tikzstyle{every node}=[circle, draw=black, fill=white, inner sep=0pt, minimum width=5pt];
    \path (160:1.2cm) node (p6) [label = left: $p_3$] {} ;
    \path (120:1.2cm) node (p1) [label = left: $p_1$] {} ;
    \path (255:1.2cm) node (p2) [label = left: $p_2$] {} ;
     \path (20:1.2cm) node (p3) [label = below right: $p_6$] {} ;
    \path (60:1.2cm) node (p4) [label = right: $p_4$] {} ;
    \path (285:1.2cm) node (p5) [label = right: $p_5$] {} ;
     \draw (p1) -- (p4);
     \draw [ultra thick, ->, black!40!white] (p6) --  (160:0.7)node[draw=white, below right=0.3pt]{};
     \draw (p1) -- (p5);
     \draw (p1) -- (p6);
     \draw (p2) -- (p4);
     \draw (p2) -- (p5);
     \draw (p2) -- (p6);
     \draw (p3) -- (p4);
     \draw (p3) -- (p5);
     \draw (p3) -- (p6);
     \draw [ultra thick, ->, black!40!white] (p1) -- node[rectangle, draw=white, above right=3pt] {}(120:1.7);
     \draw [ultra thick, ->, black!40!white] (p2) -- node[rectangle, draw=white, below right=3pt] {}(255:1.7);
    \draw [ultra thick, ->, black!40!white] (p3) -- node[rectangle, draw=white, above=4pt] {}(20:1.7);
     \draw [ultra thick, ->, black!40!white] (p4)  -- node[rectangle, draw=white, left=4pt] {} node[rectangle, draw=white, below=27pt] {}(60:0.7);
    \draw [ultra thick, ->, black!40!white] (p5) -- (285:0.7);
     \draw [dashed, thin] (0,-1.6) -- (0,1.6);
     \node [draw=white, fill=white] (c) at (0,-2.5) {(a)};
        \end{tikzpicture}
                \hspace{0.2cm}
 \begin{tikzpicture}[very thick,scale=1]
\tikzstyle{every node}=[circle, draw=black, fill=white, inner sep=0pt, minimum width=5pt];
    \path (-0.7,0.8) node (p1) [label = left: $p_{1}$]  {} ;
    \path (0.7,0.8) node (p4)[label = right: $p_{4}$] {} ;
    \path (-1.6,-0.8) node [label = below: $p_{2}$](p2)  {} ;
     \path (1.6,-0.8) node [label = below: $p_{3}$](p3)  {} ;
      \draw (p1) -- (p4);
    \draw (p1) -- (p2);
    \draw (p3) -- (p4);
    \draw (p2) -- (p3);
     \draw [dashed, thin] (0,-1.6) -- (0,1.6);
     \draw [ultra thick, ->, black!40!white] (p1) -- (-0.5,0.3);
     \draw [ultra thick, ->, black!40!white] (p4) -- (0.9,1.3);
      \draw [ultra thick, ->, black!40!white] (p2) -- (-2.2,-1);
      \draw [ultra thick, ->, black!40!white] (p3) -- (1,-0.6);
      \node [draw=white, fill=white] (b) at (0,-2.1) {(b)};
        \end{tikzpicture}
            \hspace{0.2cm}
         \begin{tikzpicture}[very thick,scale=1]
\tikzstyle{every node}=[circle, draw=black, fill=white, inner sep=0pt, minimum width=5pt];
    \path (-0.7,0.8) node (p1) [label = above: $p_{1}$]  {} ;
    \path (0.7,0.8) node (p4)[label = above: $p_{4}$] {} ;
    \path (-1.6,-0.8) node [label = below: $p_{2}$](p2)  {} ;
      \path (1.6,-0.8) node [label = below: $p_{3}$](p3)  {} ;
      \draw (p1) -- (p4);
    \draw (p1) -- (p2);
    \draw (p3) -- (p4);
    \draw (p2) -- (p3);
     \draw [dashed, thin] (0,-1.6) -- (0,1.6);
     \draw [ultra thick, ->, black!40!white] (p1) -- (-0.2,0.8);
      \draw [ultra thick, ->, black!40!white] (p4) -- (1.2,0.8);
      \draw [ultra thick, ->, black!40!white] (p2) -- (-1.1,-0.8);
      \draw [ultra thick, ->, black!40!white] (p3) -- (2.1,-0.8);
    \node [draw=white, fill=white] (b) at (0,-2.1) {(c)};
        \end{tikzpicture}
\end{center}
\vspace{-0.3cm}
\caption{Anti-symmetric infinitesimal motions of frameworks with mirror symmetry in the plane: (a), (b)  anti-symmetric infinitesimal motions; (c) an anti-symmetric trivial infinitesimal motion.}
\label{antisym}
\end{figure}
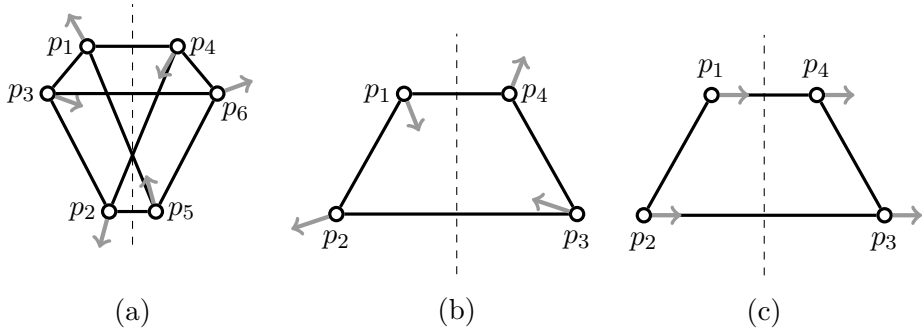

For example, consider the framework  $(G,\bp)$ shown in Fig.~\ref{antisym}(a) which is $\mathcal{C}_s$-symmetric with respect to $\theta$ and $\tau$, where $\theta:\mathcal{C}_s\to {\rm Aut}(G)$ is the action defined by $\theta(s)=(1\, 4)(2\, 5)(3\, 6)$ and $\tau:\mathcal{C}_s\to O(\mathbb{R}^2)$ is the homomorphism defined by $\tau(s)=\left(\begin{array}{c c} -1 & 0\\ 0 & 1 \end{array}\right)$. The corresponding quotient $\mathcal{C}_s$-gain graph $(H,\psi)$ is depicted in Fig.~\ref{csquotgr}, and the anti-symmetric orbit rigidity matrix $O_1(H,\psi,\tilde{\bp})$ of  $(G,\bp)$  is the following $6\times 6$ matrix:
\begin{displaymath}\bordermatrix{
                &1&2&3\cr
                (1,3;id)&\tilde{\bp}(1)-\tilde{\bp}(3) &  0 \ 0  & \tilde{\bp}(3)-\tilde{\bp}(1)\cr
                 (1,2;s)& \tilde{\bp}(1)-\tau(s)\tilde{\bp}(2) &  -(\tilde{\bp}(2)-\tau(s)^{-1}\tilde{\bp}(1)) & 0 \ 0  \cr
                (2,3;id)& 0 \ 0  & \tilde{\bp}(2)-\tilde{\bp}(3)   & \tilde{\bp}(3)-\tilde{\bp}(2) \cr
(1,1;s)& 0 \ 0  & 0 \ 0   & 0 \ 0 \cr
(2,2;s)& 0 \ 0  & 0 \ 0   & 0 \ 0 \cr
(3,3;s)& 0 \ 0  & 0 \ 0   & 0 \ 0 \cr
}
                \end{displaymath}
where an edge $(u,v)$ with label $g$ is denoted by $(u,v;g)$.

Recall from Proposition~\ref{lem:zero_loop} that each  loop in $(H,\psi)$ gives rise to a zero vector in $O_1(H,\psi,\tilde{\bp})$, and hence $O_1(H,\psi,\tilde{\bp})$ has only three non-trivial rows. Geometrically, this is also obvious, as any loop in $(H,\psi)$ clearly does not constitute any constraint if we restrict ourselves to anti-symmetric infinitesimal motions (see again Fig.~\ref{antisym}(a)).
% For example, for any displacement vector $\bmm_1$ at $\tilde{p}_1$, we have
%\begin{eqnarray*}& (\tilde{\bp}_1-\tau(s)(\tilde{\bp}_1))\bmm_1+(\tau(s)(\tilde{\bp}_1)-\tilde{\bp}_1)(-\tau(s)(\bmm_1)) \\ =& (\tilde{\bp}_1-\tau(s)(\tilde{\bp}_1))\bmm_1+(\tilde{\bp}_1-\tau(s)(\tilde{\bp}_1))(-\bmm_1)\\ & = 0.\end{eqnarray*}

\begin{figure}[htp]
\begin{center}
\begin{tikzpicture}[very thick,scale=1]
\tikzstyle{every node}=[circle, draw=black, fill=white, inner sep=0pt, minimum width=5pt];%Vertices
 \path (160:1.2cm) node (p6) [label = right: $3$] {} ;
 \path (120:1.2cm) node (p1) [label = left: $1$] {} ;
    \path (255:1.2cm) node (p2) [label = left: $2$] {} ;
%Unlabled edges
\path
(p1) edge (p6);
\path
(p2) edge (p6);
\path[<-]
(p2) edge (p1);
%Loops
\path
(p1) edge [loop above,->, >=stealth,shorten >=2pt,looseness=26] (p1);
\path
(p2) edge [loop below,->, >=stealth,shorten >=2pt,looseness=26] (p2);
\path
(p6) edge [loop left,->, >=stealth,shorten >=2pt,looseness=26] (p6);

 \node [draw=white, fill=white] (b) at (-0.05,-1.55) {$s$};
  \node [draw=white, fill=white] (b) at (-0.25,1.4) {$s$};
    \node [draw=white, fill=white] (b) at (-1.5,0.1) {$s$};
        \node [draw=white, fill=white] (b) at (-0.15,0) {$s$};
\end{tikzpicture}
\end{center}
\vspace{-0.3cm}
\caption{The $\mathcal{C}_s$-gain graph $(H,\psi)$ of the framework in Fig.~\ref{antisym}(a), where the directions and labels of edges with gain $id$ are omitted.}
\label{csquotgr}
\end{figure}
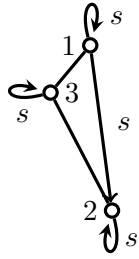

%%%%%%%%%%%%%%%%%%%%%%%%%%%%%%%%%%%%%%%%%%%%%%%%%%%%%%%%%%%%%%%%%%%%%%%%%%%%%%%%%

\subsubsection{Rotation symmetry $\mathcal{C}_{3}$}

Over the complex numbers, the symmetry group $\mathcal{C}_{3}$ has three non-equivalent one-dimensional irreducible representations. In the Mulliken notation, they are denoted by $A$, $E^{(1)}$ and $E^{(2)}$ (see Table \ref{tabc3}).

\begin{table}[htp]
\begin{center}
\begin{tabular}{c||r|r|r}
$\mathcal{C}_{3}$   &   $id$  &  $C_{3}$  &  $C_{3}^2$ \\\hline\hline
$A=\rho_0$  &    1  &  1  & 1\\\hline
$E^{(1)}=\rho_1$  &    1  &  $\omega$ & $\omega^2$ \\\hline
$E^{(2)}=\rho_2$  &    1  &  $\omega^2$ & $\omega$ \\
\end{tabular}
\end{center}
\caption{The irreducible representations of $\mathcal{C}_3$, where $\omega=\frac{2\pi \sqrt{-1}}{3}$.}
\label{tabc3}
\end{table}

It follows that the block-decomposed rigidity matrix $\widetilde{R}(G,\bp)$ of a $\mathcal{C}_3$-symmetric framework $(G,\bp)$ consists of three blocks: the submatrix block $\widetilde{R}_0(G,\bp)$  corresponding to the trivial representation $\rho_0$,
% (which is equivalent to the fully symmetric orbit rigidity matrix $O_0(H,\psi,\tilde{p})$ \cite{BSWW}),
 the submatrix block $\widetilde{R}_1(G,\bp)$ corresponding to   $\rho_1$, and the submatrix block $\widetilde{R}_2(G,\bp)$ corresponding to $\rho_2$. By Proposition~\ref{prop:cyclic}, each block $\widetilde{R}_j(G,\bp)$ is equivalent to its corresponding orbit rigidity matrix $O_j(H,\psi,\tilde{\bp})$.

As an example, consider the $\mathcal{C}_3$-symmetric framework $(G,\bp)$ shown in Figure~\ref{c3ggr}, where $\theta:\mathcal{C}_3\to {\rm Aut}(G)$ is the action defined by $\theta(C_3)=(1\, 2\, 3)(4\, 5 \, 6)$, and $\tau:\mathcal{C}_3\to O(\mathbb{R}^2)$ is the homomorphism defined by $\tau(C_3)=\left(\begin{array}{c c} -\frac{1}{2} & -\frac{\sqrt{3}}{2}\\  \frac{\sqrt{3}}{2} & -\frac{1}{2} \end{array}\right)$. Note that for this example, each of the three orbit rigidity matrices is a $3\times 4$ matrix.
%In the following, we think of $\mathcal{C}_3$ as the group $\mathbb{Z}/3\mathbb{Z}=\{0,1,2\}$.

  \begin{figure}[htp]
\begin{center}
\begin{tikzpicture}[very thick,scale=1]
       \tikzstyle{every node}=[circle, draw=black, fill=white, inner sep=0pt, minimum width=5pt];
  \node (p1) at (90:1.4cm) [label = left: $p_{1}$]  {};
\node (p2) at (210:1.4cm)  [label = left: $p_{2}$] {};
\node (p3) at (330:1.4cm) [label = right: $p_{3}$]  {};
\node (p4) at (60:0.5cm) [label = above left: $p_{4}$]  {};
\node (p5) at (180:0.5cm)  [label = below: $p_{5}$] {};
\node (p6) at (300:0.5cm)  [label = above right: $p_{6}$] {};
   \draw (p2) -- (p3);
     \draw (p1) -- (p2);
     \draw (p1) -- (p3);
     \draw (p1) -- (p4);
     \draw (p2) -- (p5);
     \draw (p3) -- (p6);
     \draw (p6) -- (p4);
     \draw (p6) -- (p5);
     \draw (p5) -- (p4);
       \node [rectangle, draw=white, fill=white] (b) at (0,-1.5) {(a)};
        \end{tikzpicture}
          \hspace{1.5cm}
     \begin{tikzpicture}[very thick,scale=1]
\tikzstyle{every node}=[circle, draw=black, fill=white, inner sep=0pt, minimum width=5pt];
%Vertices
    \path (0,0.2) node (p2) [label = right: $5$] {} ;
   \path (-0.6,-0.5) node (p1) [label = right: $2$] {} ;
%Unlabled edges
\path
(p2) edge (p1);
%Loops
\path
(p1) edge [loop left,->, >=stealth,shorten >=2pt,looseness=26] (p1);
\path
(p2) edge [loop above,->, >=stealth,shorten >=2pt,looseness=26] (p2);
\node [rectangle, draw=white, fill=white] (b) at (-1,-0.85) {$C_3$};
\node [rectangle, draw=white, fill=white] (b) at (0.52,0.58) {$C_3$};
\node [rectangle, draw=white, fill=white] (b) at (0,-1.5) {(b)};
\end{tikzpicture}
\end{center}
\vspace{-0.3cm}
\caption{A $\mathcal{C}_3$-symmetric framework and its corresponding $\mathcal{C}_3$ quotient gain graph.}
\label{c3ggr}
\end{figure}
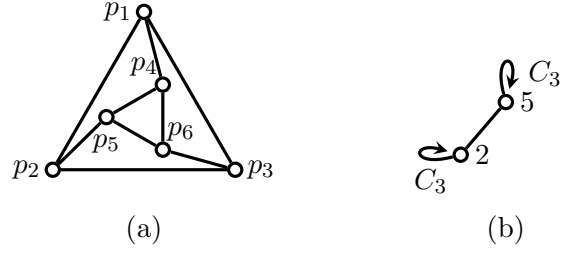

%The block $\widetilde{R}_1(G,\bp)$ describes the $\rho_1$-symmetric infinitesimal rigidity properties of $(G,\bp)$ and is equivalent to $O_1(H,\psi,\tilde{\bp})$, by Proposition~\ref{prop:cyclic}.
The orbit rigidity matrix $O_1(H,\psi,\tilde{\bp})$ is the $3\times 4$ matrix \begin{displaymath}\bordermatrix{
                &2&5\cr
                & \tilde{\bp}(2)-\tilde{\bp}(5)  & \tilde{\bp}(5)-\tilde{\bp}(2) \cr
                 & \tilde{\bp}(2)-\tau(C_3)\tilde{\bp}(2)+\omega\big(\tilde{\bp}(2)-\tau(C_3)^{-1}\tilde{\bp}(2)\big) &  0 \ 0  \cr
                & 0 \ 0   & \tilde{\bp}(5)-\tau(C_3)\tilde{\bp}(5)+\omega^2\big(\tilde{\bp}(5)-\tau(C_3)^{-1}\tilde{\bp}(5)\big) \cr},
                \end{displaymath}
                where the first row corresponds to the edge $(2,5;id)$,
the second row to the loop $(2,2;C_3)$, and the third row to the loop $(5,5;C_3)$.

%Similarly, the block $\widetilde{R}_2(G,\bp)$ describes the $\rho_2$-symmetric infinitesimal rigidity properties of $(G,\bp)$ and is equivalent to $O_2(H,\psi,\tilde{p})$, by Proposition~\ref{prop:cyclic}.
The orbit rigidity matrix $O_2(H,\psi,\tilde{p})$ is the $3\times 4$ matrix
\begin{displaymath}\bordermatrix{
                &2&5\cr
                & \tilde{\bp}(2)-\tilde{\bp}(5)  & \tilde{\bp}(5)-\tilde{\bp}(2) \cr
                 & \tilde{\bp}(2)-\tau(C_3)\tilde{\bp}(2)+\omega^2\big(\tilde{\bp}(2)-\tau(C_3)^{-1}\tilde{\bp}(2)\big) &  0 \ 0  \cr
                & 0 \ 0   & \tilde{\bp}(5)-\tau(C_3)\tilde{\bp}(5)+\omega\big(\tilde{\bp}(5)-\tau(C_3)^{-1}\tilde{\bp}(5)\big) \cr},
                \end{displaymath}
                where the first row corresponds to the edge $(2,5;id)$,
the second row to the loop $(2,2;C_3)$, and the third row to the loop $(5,5;C_3)$.

%%%%%%%%%%%%%%%%%%%%%%%%%%%%%%%%%%%%%%%%%%%%%%%%%%%%%%%%%%%%%%%%%%%%%%%%%%%%%%%%%%%%%%%%%%%%%%%%%%%%%%%%%%%%%%%%%%%%%%%

\subsubsection{Dihedral symmetry $\mathcal{C}_{2v}$}

Finally, we consider the dihedral group $\mathcal{C}_{2v}=\{id,C_2,s_h,s_v\}$ of order four which is the only non-cyclic Abelian point group in the plane. In the following, we think of $\mathcal{C}_{2v}$ as the additive group $\mathbb{Z}/2\mathbb{Z}\times \mathbb{Z}/2\mathbb{Z}$, where $id=(0,0)$, $C_2=(0,1)$, $s_h=(1,0)$, and $s_v=(1,1)$.
This group has four non-equivalent irreducible linear representations each of which is real and one-dimensional. In the Mulliken notation, these representations are denoted by $A_{1}$, $A_{2}$, $B_{1}$, and $B_{2}$ (see Table~\ref{tab:c2v}).

\begin{table}[htp]
\begin{center}
\begin{tabular}{r||r|r|r|r}
$\mathcal{C}_{2v}$   &   $id$  &  $C_{2}$   & $s_{h}$   &    $s_{v}$\\\hline\hline
$A_{1}=\rho_{(0,0)}$  &    1  &  1 &  1  &  1\\\hline
$A_{2}=\rho_{(1,0)}$  &    1  &  1 &  -1  &  -1\\\hline
$B_{1}=\rho_{(0,1)}$  &    1  &  -1 &  1  &  -1\\\hline
$B_{2}=\rho_{(1,1)}$  &    1  &  -1 &  -1  &  1\\
\end{tabular}
\end{center}
\caption{The irreducible representations of $\mathcal{C}_{2v}$.}
\label{tab:c2v}
\end{table}

Thus, for the dihedral group $\mathcal{C}_{2v}$, the block-decomposed rigidity matrix $\widetilde{R}(G,\bp)$ consists of four blocks, each of which corresponds to one of the four irreducible representations of $\mathcal{C}_{2v}$. The submatrix block corresponding to $\rho_0$ is of course again equivalent to the (fully symmetric) orbit rigidity matrix. In the following, we give an example of a $B_1$-symmetric orbit rigidity matrix $O_{(0,1)}(H,\psi,\tilde{\bp})$ which, by Proposition~\ref{prop:abelian}, is equivalent to its corresponding submatrix block $\widetilde{R}_{(0,1)}(G,\bp)$.

Consider the $\mathcal{C}_{2v}$-symmetric framework $(G,\bp)$ shown in Figure~\ref{k44fl}(a), where $\theta:\mathcal{C}_{2v}\to {\rm Aut}(G)$ is the action defined by $\theta(s_h)=(1\, 4)(2\, 3)(5\, 8)(6 \, 7)$ and $\theta(s_v)=(1\, 2)(3\, 4)(5\, 6)(7 \, 8)$, and $\tau:\mathcal{C}_{2v}\to O(\mathbb{R}^2)$ is the homomorphism defined by $\tau(s_h)=\left(\begin{array}{c c} 1 & 0\\ 0 & -1 \end{array}\right)$ and $\tau(s_v)=\left(\begin{array}{c c} -1 & 0\\ 0 & 1 \end{array}\right)$.

\begin{figure}[htp]
\begin{center}
\begin{tikzpicture}[very thick,scale=1]
\tikzstyle{every node}=[circle, draw=black, fill=white, inner sep=0pt, minimum width=5pt];
        \path (-2.4,-0.6) node (p1) [label =  left: $p_4$] {} ;
        \path (2.4,-0.6) node (p2) [label =  right: $p_3$] {} ;
        \path (2.4,0.6) node (p3) [label =  right: $p_2$] {} ;
        \path (-2.4,0.6) node (p4) [label =  left: $p_1$] {} ;
        \path (-1.6,-1.2) node (p5) [label = left: $p_8$] {} ;
        \path (1.6,-1.2) node (p6) [label =  right: $p_7$] {} ;
        \path (1.6,1.2) node (p7) [label =  right: $p_6$] {} ;
        \path (-1.6,1.2) node (p8) [label =  left: $p_5$] {} ;
        \draw (p1) -- (p5);
        \draw (p1) -- (p6);
        \draw (p1) -- (p7);
        \draw (p1) -- (p8);
        \draw (p2) -- (p5);
        \draw (p2) -- (p6);
        \draw (p2) -- (p7);
        \draw (p2) -- (p8);
        \draw (p3) -- (p5);
        \draw (p3) -- (p6);
        \draw (p3) -- (p7);
        \draw (p3) -- (p8);
        \draw (p4) -- (p5);
        \draw (p4) -- (p6);
        \draw (p4) -- (p7);
        \draw (p4) -- (p8);
        \draw [->, ultra thick, black!40!white] (p4) -- (-1.9,0.23);
        \draw [->, ultra thick, black!40!white] (p3) -- (1.9,0.23);
        \draw [->, ultra thick, black!40!white] (p2) -- (1.9,-0.23);
        \draw [->, ultra thick, black!40!white] (p1) -- (-1.9,-0.23);
        \draw [->, ultra thick, black!40!white] (p8) -- (-1.85,1.85);
        \draw [->, ultra thick, black!40!white] (p7) -- (1.85,1.85);
        \draw [->, ultra thick, black!40!white] (p6) -- (1.85,-1.85);
        \draw [->, ultra thick, black!40!white] (p5) -- (-1.85,-1.85);
           \draw [dashed, thin] (0,-2) -- (0,2);
\draw [dashed, thin] (-3.3,0) -- (3.3,0);
\node [rectangle, draw=white, fill=white] (b) at (0,-2.1) {(a)};
       \end{tikzpicture}
       \hspace{1.5cm}
       \begin{tikzpicture}[very thick,scale=1]
\tikzstyle{every node}=[circle, draw=black, fill=white, inner sep=0pt, minimum width=5pt];
       \path (-0.9,-0.9) node (p4) [label =  left: $1$] {} ;
       \path (0.9,0.7) node (p8) [label =  right: $5$] {} ;
       \draw (p4) -- (p8);
       \path
(p4) edge [->,bend right=22] (p8);
  \path
(p4) edge [->,bend right=90] (p8);
  \path
(p4) edge [->,bend left=22] (p8);

\node [rectangle, draw=white, fill=white] (b) at (-0.65,0.17) {$C_2$};
\node [rectangle, draw=white, fill=white] (b) at (0.38,-0.5) {$s_h$};
\node [rectangle, draw=white, fill=white] (b) at (1.1,-0.6) {$s_v$};
\node [rectangle, draw=white, fill=white] (b) at (0,-2.1) {(b)};
       \end{tikzpicture}
       \end{center}
\vspace{-0.3cm}
\caption{A framework in $\mathcal{R}_{(K_{4,4},\mathcal{C}_{2v})}$ with a fully symmetric infinitesimal flex (a)  and its corresponding quotient $\mathcal{C}_{2v}$-gain graph (b), where the direction and label of the edge with gain $id$ is omitted.}
\label{k44fl}
\end{figure}
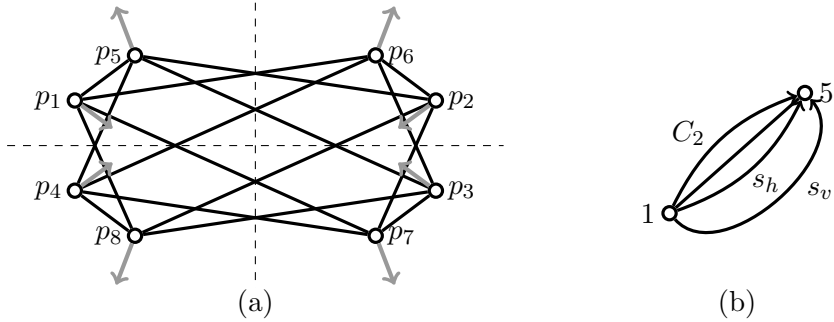

The $B_1$-symmetric orbit rigidity matrix $O_{(0,1)}(H,\psi,\tilde{\bp})$ of $(G,\bp)$ is the $4\times 4$ matrix
\begin{displaymath}\bordermatrix{
                &1&5\cr
                (1,5)&\tilde{\bp}(1)-\tilde{\bp}(5)  &  \tilde{\bp}(5)-\tilde{\bp}(1)\cr
                 (1,5;C_2) & \tilde{\bp}(1)-\tau(C_2)\tilde{\bp}(5) &  -\big(\tilde{\bp}(5)-\tau(C_2)^{-1}\tilde{\bp}(1)\big) \cr
                 (1,5;s_h) & \tilde{\bp}(1)-\tau(s_h)\tilde{\bp}(5) &  \tilde{\bp}(5)-\tau(s_h)^{-1}\tilde{\bp}(1) \cr
                 (1,5;s_v) & \tilde{\bp}(1)-\tau(s_v)\tilde{\bp}(5) &  -\big(\tilde{\bp}(5)-\tau(s_v)^{-1}\tilde{\bp}(1)\big) \cr}.
                              \end{displaymath}

The other orbit rigidity matrices $O_{\bm j}(H,\psi,\tilde{\bp})$
%(which are equivalent to the submatrix blocks $\widetilde{R}_{\bm j}(G,\bp)$, by Proposition~\ref{prop:abelian})
can be obtained analogously.

Note that the framework in Figure~\ref{k44fl}(a) has a non-trivial fully symmetric infinitesimal motion which even extends to a \emph{continuous} $\mathcal{C}_{2v}$-preserving motion \cite{BSWW,jkt}. (In the engineering literature, this motion is known as Bottema's mechanism.)  It was shown in \cite{jkt} that this framework is falsely predicted to be forced-symmetric rigid by the matroidal counts for the fully symmetric orbit rigidity matrix. Thus, the problem of finding combinatorial characterizations for forced-symmetric rigidity (and hence also for incidentally symmetric rigidity)
of ${\cal C}_{2v}$-generic frameworks  (or $\mathcal{C}_{2nv}$-generic frameworks, $n\geq 1$) remains open.

%${\bm i}=(i_1,\dots,i_l)$, where  $0\leq i_1\leq k_1,\dots, 0\leq i_l\leq k_l$,

%For an edge $\{a,b\}$ we always write the standard row from the rigidity matrix:
% \begin{displaymath}\renewcommand{\arraystretch}{0.8}
%        \bordermatrix{  & &  a &  &  b &  \cr & 0 \ldots 0 & \big(p_a-p_b\big) & 0  \ldots  0 & \big(p_b-p_a\big) & 0  \ldots  0}\textrm{.}
%    \end{displaymath}

%For an edge $\{a,x(b)\}$ (where $x$ is either the half-turn or one of the two reflections), we write the row of the symmetric orbit matrix for the group $\{Id,x\}$ if the irreducible representation has a $1$ under column $x$ in the character table, and we write the row of the anti-symmetric orbit matrix for the group $\{Id,x\}$ if the irreducible representation has a $-1$ under column $x$.

%Problem: No known characterization for forced symmetry for $\mathcal{C}_{2v}$ (see, for example the example in the Figure (Bottema's mechanism)....)

%%%%%%%%%%%%%%%%%%%%%%%%%%%%%%%%%%%%%%%%%%%%%%%%%%%%%%%%%%%%%%%%%%%%%%%%%%%%%%%%%%%%%%%%%%%%%%%%%%%%%%%%%%%%%%%%%%%%%%%%%%%%%

\section{Gain-sparsity and constructive characterizations}
\label{sec:gain_sparsity}
We now turn our attention to combinatorial characterizations of infinitesimally rigid symmetric frameworks in the plane.
In this section we first present some preliminary facts concerning gain graphs and
matroids on gain graphs which will be used in
the next section to derive the desired combinatorial characterizations.

\subsection{Gain-sparsity}
Let $(H,\psi)$ be a $\Gamma$-gain graph.
A cycle is called {\em balanced} if the product of its edge gains is equal to the identity.
(If $\Gamma$ is an additive group, we take the sum instead of the product.)
More precisely, a cycle of the form $\tilde{v}_1,\te_1,\tilde{v}_2,\te_2,\tilde{v}_3,\dots,\tilde{v}_k,\te_k,\tilde{v}_1$, is balanced if $\Pi_{i=1}^k \psi(\te_i)^{{\rm sign}(\te_i)}=id$,
where ${\rm sign}(\te_i)=1$ if $\te_i$ is directed from $\tilde{v}_i$ to $\tilde{v}_{i+1}$, and ${\rm sign}(\te_i)=-1$  otherwise.

We say that an edge subset $F\subseteq E(H)$ is {\em balanced} if all cycles in $F$ are balanced;
otherwise it is called {\em unbalanced}.
The following is a slightly generalized concept to the one proposed in \cite{jkt}.
\begin{definition}
Let $(H,\psi)$ be a $\Gamma$-gain graph and $k, \ell, m$ be nonnegative integers with $m\leq \ell$.
$(H,\psi)$ is called $(k,\ell,m)$-gain-sparse if
\begin{itemize}
\item $|F|\leq k|V(F)|-\ell$ for any nonempty balanced $F\subseteq E(H)$;
\item $|F|\leq k|V(F)|-m$ for any nonempty $F\subseteq E(H)$.
\end{itemize}
Similarly, an edge set $E$ is called $(k,\ell,m)$-gain-sparse if it induces a $(k,\ell,m)$-gain-sparse graph.
\end{definition}

Let ${\cal I}_{k,\ell,m}$ be a family of $(k,\ell,m)$-gain-sparse edge sets in $(H,\psi)$.
As noted in \cite{jkt}, ${\cal I}_{k,\ell,m}$ forms the family of independence sets
of a  matroid on $E(H)$ for certain $(k,\ell, m)$,
which we denote by ${\cal M}_{k,\ell,m}(H,\psi)$, or simply by ${\cal M}_{k,\ell,m}$.
Let us take a closer look at this fact.

If $(k,\ell,m)=(1,1,0)$, then ${\cal M}_{1,1,0}$ is known as the {\em frame matroid} (or {\em bias matroid}) of $(H,\psi)$, which is extensively studied in matroid theory (see, e.g., \cite{zaslavsky1991biased}).
It is known that $F\subseteq E(H)$ is independent in ${\cal M}_{1,1,0}$ if and only if each connected component of $F$ contains no cycle or just one cycle, and the cycle is unbalanced if it exists.
When $\Gamma=\{id\}$, ${\cal M}_{1,1,0}$ is equal to the {\em graphic matroid} of $H$,
where $F\subseteq E(H)$ is independent if and only if $F$ is cycle free.

If $k=\ell$, ${\cal M}_{k,k,m}$ is the union of
$m$ copies of the graphic matroid of $H$ and
$(k-m)$ copies of the frame matroid of $(H,\psi)$.
In other words, $F\subseteq E(H)$ is independent in ${\cal M}_{k,k,m}$ if and only if
$F$ can be partitioned into $k$ sets $F_1,\dots, F_k$ such that
$F_i$ is a forest for $1\leq i\leq m$ and
$F_i$ is independent in ${\cal M}_{1,1,0}$ for $m+1\leq i\leq k$.
In particular, if $|E(H)|=k|V(H)|-m$, then $E(H)$ can be partitioned into $k$ sets $E_1,\dots, E_k$ such that
$E_i$ is a spanning tree for $1\leq i\leq m$ and
$E_i$ is a spanning edge set such that each connected component contains exactly one unbalanced cycle.

If $(k,\ell,m)=(k,k+\ell',m'+\ell')$ for some $0\leq m'\geq k$ and $\ell'\geq 0$, then ${\cal M}_{k,\ell,m}$ is $\ell'$ times Dilworth truncations
of ${\cal M}_{k,k,m'}$, and it forms a matroid.
In particular, for $k=2$ and $\ell=3$, ${\cal M}_{2,3,m}$ implicitly or explicitly appeared in the study of symmetry-forced rigidity.
The generic symmetry-forced rigidity of ${\cal C}_s$-symmetric frameworks or ${\cal C}_k$-symmetric frameworks is characterized by
the $(2,3,1)$-gain-sparsity of the underlying quotient gain graphs \cite{MT3,MT2,MT1,LT,jkt}.
We shall extend this result in Section~\ref{sec:characterization}.
For infinite periodic graphs, it was proved by Ross that the $(2,3,2)$-gain-sparsity of $\mathbb{Z}^2$-gain graphs characterizes the symmetry-forced rigidity of
periodic frameworks on a fixed lattice \cite{ER}.

For other triples $(k,\ell,m)$ very little properties are known for $(k,\ell,m)$-gain-sparse graphs.
Csaba Kiraly recently pointed out that ${\cal M}_{2,3,0}$ is not a matroid in general.
A number of different (or generalized) sparsity conditions of gain graphs are also discussed in \cite{MT1,MT3,jkt,tan}.

\subsection{Constructive characterizations of $(2,3,m)$-gain-sparse graphs}
In this subsection we will review a constructive characterization of $(2,3,m)$-gain-sparse graphs given in \cite{jkt}.
We define three operations,
called {\it extensions}, that preserve $(2,3,m)$-gain-sparsity.
The first two operations
generalize the well-known Henneberg operations~\cite{W1} to gain graphs.

Let $(H,\psi)$ be a $\Gamma$-gain graph.
%We define three graph opeartions called {\em extension} as follows.
The \emph{0-extension} adds a new vertex \(\tilde{v}\) and two new non-loop edges
$\te_{1}$ and $\te_{2}$ to \(H\)
such that the new edges are incident to \(\tilde{v}\) and the other end-vertices are
two not necessarily distinct vertices of \(V(H)\).
If $\te_{1}$ and $\te_{2}$ are not parallel, then their labels can be arbitrary.
Otherwise the labels are assigned such that \(\psi(\te_{1})\neq \psi(\te_{2})\),
assuming that $\te_1$ and $\te_2$ are directed to $\tilde{v}$ (see Fig.\ref{fig:inductive} (a)).

The \emph{1-extension} (see Fig.\ref{fig:inductive} (b)) first chooses an edge \(\te\)  and a vertex \(\tilde{z}\),
where $\te$ may be a loop and $\tilde{z}$ may be an end-vertex of $\te$.
It subdivides \(\te\), with a new vertex \(\tilde{v}\) and new edges \(\te_{1},\te_{2}\),
such that the tail of \(\te_1\) is the tail of \(\te\) and the tail of \(\te_2\) is the head of \(\te\).
The labels of the new edges are assigned such that \(\psi(\te_{1})\cdot \psi(\te_{2})^{-1}=\psi(\te)\).
The 1-extension also adds a third edge \(\te_{3}\) oriented from $\tilde{z}$ to \(\tilde{v}\).
The label of $\te_3$ is assigned so that it is {\em locally unbalanced},
i.e., every two-cycle $\te_i\te_j$, if it exists,  is unbalanced.

The \emph{loop 1-extension} (see Fig.\ref{fig:inductive} (c)). adds a new vertex \(\tilde{v}\) to \(H\)
and connects it to a vertex \(\tilde{z}\in V(H)\) by a new edge with any label.
It also adds a new loop \(\tilde{l}\) incident to \(\tilde{v}\) with \(\psi(\tilde{l})\neq id\).

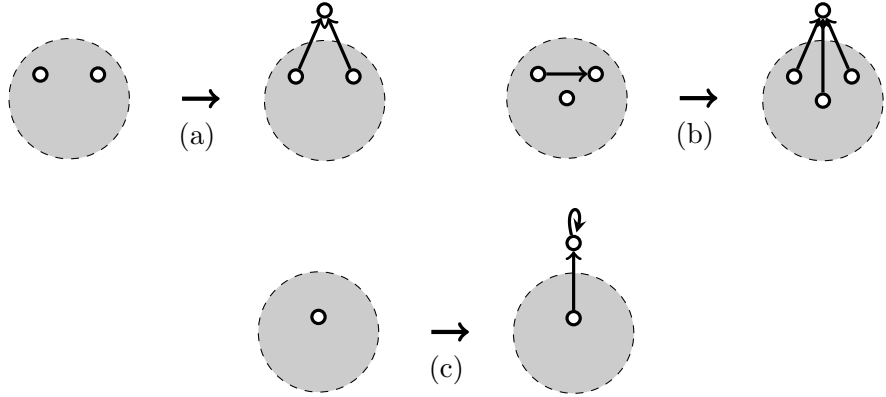
\begin{figure}[htp]
\begin{center}
\begin{tikzpicture}[very thick,scale=1]
\tikzstyle{every node}=[circle, draw=black, fill=white, inner sep=0pt, minimum width=5pt];
\filldraw[fill=black!20!white, draw=black, thin, dashed](0,0)circle(0.8cm);
\node (p1) at (40:0.5cm) {};
\node (p2) at (140:0.5cm) {};
  \draw[ultra thick, ->](1.5,0)--(2,0);
   \node [draw=white, fill=white] (a) at (1.7,-0.5)  {(a)};
\end{tikzpicture}
\hspace{0.3cm}
\begin{tikzpicture}[very thick,scale=1]
\tikzstyle{every node}=[circle, draw=black, fill=white, inner sep=0pt, minimum width=5pt];
\filldraw[fill=black!20!white, draw=black, thin, dashed](0,0)circle(0.8cm);
\node (p1) at (40:0.5cm) {};
\node (p2) at (140:0.5cm) {};
\node (p3) at (90:1.2cm) {};
\draw [<-](p3)--(p1);
\draw [<-](p3)--(p2);
\end{tikzpicture}
%%%%%
\hspace{1.35cm}
%%%%%%%
\begin{tikzpicture}[very thick,scale=1]
\tikzstyle{every node}=[circle, draw=black, fill=white, inner sep=0pt, minimum width=5pt];
\filldraw[fill=black!20!white, draw=black, thin, dashed](0,0)circle(0.8cm);
\node (p1) at (40:0.5cm) {};
\node (p2) at (140:0.5cm) {};
\node (p4) at (90:0cm) {};
 \draw [<-](p1)--(p2);
\draw[ultra thick, ->](1.5,0)--(2,0);
   \node [draw=white, fill=white] (a) at (1.7,-0.5)  {(b)};
\end{tikzpicture}
\hspace{0.3cm}
\begin{tikzpicture}[very thick,scale=1]
\tikzstyle{every node}=[circle, draw=black, fill=white, inner sep=0pt, minimum width=5pt];
\filldraw[fill=black!20!white, draw=black, thin, dashed](0,0)circle(0.8cm);
\node (p1) at (40:0.5cm) {};
\node (p2) at (140:0.5cm) {};
\node (p3) at (90:1.2cm) {};
\node (p4) at (90:0cm) {};
\draw [<-](p3)--(p1);
\draw [<-](p3)--(p2);
\draw [<-](p3)--(p4);
\end{tikzpicture}
\vspace{0.5cm}

\begin{tikzpicture}[very thick,scale=1]
\tikzstyle{every node}=[circle, draw=black, fill=white, inner sep=0pt, minimum width=5pt];
\filldraw[fill=black!20!white, draw=black, thin, dashed](0,0)circle(0.8cm);
\node (p4) at (90:0.2cm) {};
\draw[ultra thick, ->](1.5,0)--(2,0);
   \node [draw=white, fill=white] (a) at (1.7,-0.5)  {(c)};
\end{tikzpicture}
\hspace{0.3cm}   \begin{tikzpicture}[very thick,scale=1]
\tikzstyle{every node}=[circle, draw=black, fill=white, inner sep=0pt, minimum width=5pt];
\filldraw[fill=black!20!white, draw=black, thin, dashed](0,0)circle(0.8cm);
\node (p3) at (90:1.2cm) {};
\node (p4) at (90:0.2cm) {};
\draw [<-](p3)--(p4);
\path
(p3) edge [loop above,->, >=stealth,shorten >=1pt,looseness=20] (p3);
\end{tikzpicture}
\vspace{-0.3cm}
\end{center}
\vspace{-0.2cm}
\caption{(a) 0-extension, where the  new edges may be parallel.
(b) 1-extension, where the removed edge may be a loop and the
new edges may be parallel.
(c) loop-1-extension.}
\label{fig:inductive}
\end{figure}

%\begin{figure}[htp]
%\centering
%\begin{minipage}{0.43\textwidth}
%\centering
%\includegraphics[trim= 90 320 0 320, clip, scale=0.5]{0-extension}
%\par
%(a)
%\end{minipage}
%\begin{minipage}{0.43\textwidth}
%\centering
%\includegraphics[trim= 90 460 0 180, clip, scale=0.5]{1-extension}
%\par
%(b)
%\end{minipage}
%\begin{minipage}{0.43\textwidth}
%\centering
%\includegraphics[trim= 90 620 0 0, clip, scale=0.5]{loop-1-extension}
%\par
%(c)
%\end{minipage}
%\caption{(a) 0-extension, where the  new edges may be parallel.
%(b) 1-extension, where the removed edge may be a loop and the
%new edges may be parallel.
%(c) loop-1-extension.}
%\end{figure}

\begin{thm}[Jord{\'a}n et al.~\cite{jkt}]
\label{thm:construction}
Let $m\in \{1,2\}$ and let $(H,\psi)$ be a $\Gamma$-gain graph with $|E(H)|=2|V(H)|-m$.
Then $(H,\psi)$ is $(2,3,m)$-gain-sparse if and only if it can be built up
from a $\Gamma$-gain graph with one vertex without any edge if $m=2$ and with
an unbalanced loop if $m=1$ by a sequence of 0-extensions, 1-extensions, and loop-1-extensions.
\end{thm}
The theorem is proved for $m=1$ in \cite[Theorem 4.4]{jkt}, and exactly the same proof can be applied in  the case of $m=2$.
For special cases, Theorem~\ref{thm:construction} was proved
by Schulze \cite{BS3} and Ross \cite{ER}.

In the covering graph these operations can be seen as
graph operations that preserve the underlying symmetry.
Some of them can be recognized as
performing standard - non-symmetric - Henneberg operations~\cite{W1}
simultaneously \cite{jkt}.

%%%%%%%%%%%%%%%%%%%%%%%%%%%%%%%%%%%%%%%%%%%%%%%%%%%%%%%%%%%%%%%%%%%%%%%%%%%%%%%%%%%%%%%%%%%%%%%%%%%%%%%%%%%%%%%%%%%%%%%%%%%%%
\subsection{Subgroups induced by edge sets}
We have introduced the {\em balancedness} of an edge set in $(H,\psi)$
in order to define gain-sparsity matroids on $E(H)$.
However, we sometimes need to extract more information on the underlying group from $(H,\psi)$.
Such information is represented as {\em subgroups induced by edge sets}, which we are about to introduce.
For simplicity, we will assume that $\Gamma$ is Abelian. (See \cite{jkt} for the general treatment.)
%
%Suppose that $\Gamma$ is an Abelian group and let $(H,\psi)$ be a $\Gamma$-gain graph.

Recall that for a cycle $C$ of the form $\tilde{v}_1,\tilde{e}_1,\tilde{v}_2,\dots, \tilde{e}_k,\tilde{v}_1$
in $(H,\psi)$,
the gain $\psi(C)$ of $C$ is $\psi(C)=\Pi_{i=1}^k\psi(\tilde{e}_i)^{{\rm sign}(\tilde{e}_i)}$.
For $F\subseteq E(H)$, define $\langle F\rangle$ to be
the subgroup of $\Gamma$ generated by the elements
in the set $\{\psi(C)|\,\textrm{  $C$ is a cycle in the subgraph induced by }  F \}$.
 %$\psi(C)$ over all cycles $C$ contained in the subgraph induced by  $F$.
Note that $F$ is balanced if and only if $\langle F\rangle$ is trivial.

A {\em  switching} at a vertex $\tilde{v}$ with $\gamma\in \Gamma$ is an operation that constructs
a new labeling $\psi':E(H)\rightarrow \Gamma$ from $\psi$ by setting
\begin{equation*}
\psi'(\tilde{e})=\begin{cases}
\gamma \psi(\tilde{e}) & \text{if $\tilde{e}$ is directed to $\tilde{v}$} \\
\psi(\tilde{e})\gamma^{-1} & \text{if $\tilde{e}$ is directed from $\tilde{v}$} \\
\psi(\tilde{e}) & \text{otherwise}
\end{cases}
\end{equation*}
We say that $\psi'$ is {\em equivalent} to $\psi$ if $\psi'$ can be obtained from $\psi$
by a sequence of switchings.
Then it can easily be checked that for any $F\subseteq E(H)$,
$\langle F\rangle$ is invariant up to equivalence (see, e.g., \cite[Proposition 2.2]{jkt} for the proof).

In the proof of \cite[Lemma 5.2]{jkt}, it was shown that
the rank of fully-symmetric orbit rigidity matrices
(i.e., the case when $\rho_{\bm j}$ is trivial) is invariant up to equivalence.
Exactly the same proof can be applied to show the following.
\begin{proposition}
\label{prop:switching}
Let $(H,\psi)$ be a $\Gamma$-gain graph with Abelian group $\Gamma$, let $\tilde{\bp}:V(H)\rightarrow \mathbb{R}^d$
be $\Gamma$-generic,
and let $\psi'$ be a gain function equivalent to $\psi$.
Then
${\rm rank}\ O_{\bm j}(H,\psi,\tilde{\bp})={\rm rank}\ O_{\bm j}(H,\psi',\tilde{\bp})$.
\end{proposition}

The following proposition is very useful to compute $\langle F\rangle$.
\begin{proposition}
\label{prop:induced_subgroup}
Let $(H,\psi)$ be a $\Gamma$-gain graph with Abelian $\Gamma$.
\begin{itemize}
\item For any forest $T$ in $E(H)$, there exists a $\psi'$ equivalent to $\psi$ such that
$\psi'(\tilde{e})=id$ for every $\tilde{e}\in T$.
\item For any $F\subseteq E(H)$ and a maximal forest $T$ in $F$, if $\psi(\tilde{e})=id$ holds for every $\tilde{e}\in T$,
then $\langle F\rangle$ is the subgroup generated by $\{\psi(\tilde{e})|\, \tilde{e}\in  F\setminus T\}$.
%$\psi(e)$ over $e\in F\setminus T$.
\end{itemize}
\end{proposition}
The proof is given in \cite[Proposition 2.3, Lemma 2.4]{jkt}.

\section{Combinatorial characterizations for bar-joint frameworks in the plane}
\label{sec:characterization}
Based on the theory of block-diagonalizations of rigidity matrices,
in this section we present combinatorial characterizations of infinitesimally rigid frameworks which are generic modulo cyclic symmetry in the plane.
By (\ref{rigblocks}) and
Proposition~\ref{prop:abelian} our task of computing the rank of the rigidity matrix is reduced to computing the rank of each orbit rigidity matrix.

Recall that each orbit rigidity matrix is defined for any $\Gamma$-gain graph $(H,\psi)$ with $\tilde{\bp}:V(H)\rightarrow \mathbb{R}^d$, and
its rows define a matroid on the edge set of $H$.
We will show that  when $\tilde{\bp}$ is $\tau(\Gamma)$-regular,
this {\em orbit-rigidity matroid} is isomorphic to the $(2,3,m)$-gain-sparsity matroid of $(H,\psi)$
given in Section~\ref{sec:gain_sparsity}
if the underlying symmetry is ${\cal C}_s, {\cal C}_2$ or ${\cal C}_3$.
If the underlying symmetry is ${\cal C}_k$ for $k\geq 4$, then it turns out that orbit-rigidity matroids have more complicated combinatorial structures and
the problem of characterizing them is still unsolved.
However, we will present some non-trivial necessary conditions in the last subsection.

The following lemma implies that the row independence of an orbit rigidity matrix
is preserved by the three operations given in
Section~\ref{sec:gain_sparsity}.
\begin{lemma}
\label{lem:extension}
Let $\Gamma$ be an Abelian group of order $k$ and let
$\tau:\Gamma\rightarrow O(\mathbb{R}^2)$ be a faithful orthogonal representation.
Let $(H',\psi')$ be a $\Gamma$-gain graph obtained from $(H,\psi)$ by a 0-extension, 1-extension, or
loop-1-extension.
If there is a mapping $\tilde{\bp}:V(H)\rightarrow \mathbb{R}^2$ such that
$O_{\bm j}(H,\psi,\tilde{\bp})$ is row independent, then there is $\tilde{\bp}':V(H')\rightarrow \mathbb{R}^2$
such that $O_{\bm j}(H',\psi,\tilde{\bp}')$ is row independent, unless
the new loop is a zero loop in the case of a loop-1-extension.
\end{lemma}
\begin{proof}
The proof is basically the same as the one given in \cite[Lemma 6.1]{jkt} for symmetry-forced rigidity.
Due to the definition of genericity, we may assume that $\tilde{\bp}$ is $\Gamma$-generic.
Then it is easy to prove the statement for a $0$-extension and a loop-$1$-extension
(see the proof of \cite[Lemma 6.1]{jkt} for a formal proof).
We therefore focus on the case, where $H'$ is obtained from $H$ by a 1-extension. This is the only nontrivial case.

Suppose that $H'$ is obtained from $H$ by a 1-extension
which removes an existing edge $\te$ and adds a new vertex $\tilde{v}$
with three new non-loop edges $\te_1, \te_2, \te_3$ incident to $\tilde{v}$.
We may assume that $\te_i$ is outgoing from $\tilde{v}$.
Let $\tilde{u}_i$ be the other end-vertex of $\te_i$,
and let $g_i=\tau(\psi'(\te_i))$ and $\tilde{p}_i=\tilde{\bp}(\tilde{u}_i)$ for $i=1,2,3$.
By the definition of the 1-extension,
we have $\tau(\psi(\te))=g_1^{-1}g_2$.
We also denote $\omega_i=\rho_{\bm j}(\psi'(\te_i))$ for $i=1,2,3$.

Note that
the three points $g_i\tilde{p}_i \ (i=1,2,3)$ never lie on a line due to the $\Gamma$-genericity of $\tilde{\bp}$
(see \cite[Lemma 6.1]{jkt} for a formal proof).
%\label{claim:cyclic_rigidity}
%\end{claim}
%\begin{proof}
%Since $p$ is ${\cal S}$-generic, $u_1=u_2=u_3$ should hold  if they lie on a line.
%Then $p_1=p_2=p_3$.
%By the definition of 1-extensions, $g_i\neq g_j$ if $u_i=u_j$.
%This implies that $g_1p_1, g_2p_2, g_3p_3$ are three distinct points on a circle.
%Thus, they do not lie on a line.
%\end{proof}
We take $\tilde{\bp}':V(H')\rightarrow \mathbb{R}^2$ such that $\tilde{\bp}'(w)=\tilde{\bp}(w)$ for all $\tilde{w}\in V(H)$,
and $\tilde{\bp}'(\tilde{v})$ is a point on the line through $g_1\tilde{p}_1$ and $g_2\tilde{p}_2$, but distinct from $g_1\tilde{p}_1$ and $g_2\tilde{p}_2$.
For the simplicity of the description, we assume $\tilde{u}_1\neq \tilde{u}_2$ in the subsequent discussion,
but exactly the same proof can be also applied if $\tilde{u}_1=\tilde{u}_2$.
Then $O_{\bm j}(H',\psi',\tilde{\bp}')$ has the form
\begin{equation*}
\begin{array}{c|c|c|c|cc|}
\multicolumn{1}{c}{} & \multicolumn{1}{c}{\tilde{v}} & \multicolumn{1}{c}{\tilde{u}_1} & \multicolumn{1}{c}{\tilde{u}_2}    \\ \cline{2-6}
\te_3 & \tilde{\bp}'(\tv)-g_3\tilde{p}_3 & \ast & \ast & \ast & \\
\te_1 & \tilde{\bp}'(\tv)-g_1\tilde{p}_1 & \omega_1(\tilde{p}_1-g_1^{-1}\tilde{\bp}'(\tilde{v})) & 0 &  0 &\\
\te_2 & \tilde{\bp}'(\tilde{v})-g_2\tilde{p}_2 & 0 & \omega_2(\tilde{p}_2-g_2^{-1}\tilde{\bp}'(\tilde{v})) &  0 & \\ \cline{2-6}
E(H)-\te  & 0 & \multicolumn{3}{c}{O_{\bm j}(H-\te,\psi,\tilde{\bp})} &  \\ \cline{2-6}
\end{array}
\end{equation*}
where the bottom right block $O_{\bm j}(H-\te,\psi,\tilde{\bp})$ denotes the $\rho_{\bm j}$-orbit rigidity matrix obtained from
$O_{\bm j}(H,\psi,\tilde{\bp})$ by removing the row of $\te$.

Since $\tilde{\bp}'(v)$ lies on the line through $g_1\tilde{p}_1$ and $g_2\tilde{p}_2$,
$\tilde{\bp}'(\tilde{v})-g_i\tilde{\bp}(\tilde{u}_i)$ is a scalar multiple of $g_1\tilde{p}_1-g_2\tilde{p}_2$ for $i=1,2$.
Hence, by multiplying  the rows of $\te_1$ and $\te_2$ by an appropriate scalar, $O(H',\psi',\tilde{\bp}')$ becomes
\begin{equation*}
\begin{array}{c|c|c|c|cc|}
\multicolumn{1}{c}{} & \multicolumn{1}{c}{\tilde{v}} & \multicolumn{1}{c}{\tilde{u}_1} & \multicolumn{1}{c}{\tilde{u}_2}    \\ \cline{2-6}
\te_3 & \tilde{\bp}'(\tilde{v})-g_3\tilde{p}_3 & \ast & \ast & \ast & \\
\te_1 & g_1\tilde{p}_1-g_2\tilde{p}_2 & -\omega_1g_1^{-1}(g_1\tilde{p}_1-g_2\tilde{p}_2) & 0 &  0 &\\
\te_2 & g_1\tilde{p}_1-g_2\tilde{p}_2 & 0 & -\omega_2g_2^{-1}(g_1\tilde{p}_1-g_2\tilde{p}_2) &  0 & \\ \cline{2-6}
E(H)-\te  & 0 & \multicolumn{3}{c}{O_{\bm j}(H-\te,\psi,\tilde{\bp})} &  \\ \cline{2-6}
\end{array}
\end{equation*}
Subtracting the row of $\te_1$ from that of $\te_2$, we  get
\begin{equation*}
\begin{array}{c|c|c|c|cc|}
\multicolumn{1}{c}{} & \multicolumn{1}{c}{\tilde{v}} & \multicolumn{1}{c}{\tilde{u}_1} & \multicolumn{1}{c}{\tilde{u}_2}    \\ \cline{2-6}
\te_3 & \tilde{\bp}'(\tilde{v})-g_3\tilde{p}_3 & \ast & \ast & \ast & \\
\te_1 & g_1\tilde{p}_1-g_2\tilde{p}_2 & -\omega_1g_1^{-1}(g_1\tilde{p}_1-g_2\tilde{p}_2) & 0 &  0 &\\ \cline{2-6}
\te_2 & 0 & \tilde{p}_1-g_1^{-1}g_2\tilde{p}_2 & \omega_1^{-1}\omega_2(\tilde{p}_2-g_2^{-1}g_1\tilde{p}_1) &  0 & \\ \cline{3-6}
E(H)-\te  & 0 & \multicolumn{3}{c}{O(H-\te,\psi,\tilde{\bp})} &  \\ \cline{2-6}
\end{array}
\end{equation*}
Since $\tau(\psi(\te))=g_1^{-1}g_2$, the row of $\te_2$ is equal to the row of $\te$ in $O_{\bm j}(H,\psi,\tilde{\bp})$.
This means that the right-bottom block together with the row of $\te_2$ forms
$O_{\bm j}(H,\psi,\tilde{\bp})$, which is row independent.
Thus, the matrix is row independent if and only if the top-left block is row independent.
Since $g_i\tilde{p}_i \ (i=1,2,3)$ are not on a line,
the line through $\tilde{\bp}'(v)$ and $g_3\tilde{p}_3$ is not parallel to
the line through $g_1\tilde{p}_1$ and $g_2\tilde{p}_2$.
This implies that the top-left $2\times 2$-block is row independent, and consequently $O_{\bm j}(H',\psi',\tilde{\bp}')$ is
row independent.
\end{proof}

%%%%%%%%%%%%%%%%%%%%%%%%%%%%%%%%%%%%%%%%%%%%%%%%%%%%%%%%%%%%%%%%%%%%%%%%%%%%%%%%%%%%%%%%%%%%%%%%%%%%%%%%%%%%%%%%%%%%%%%%%%%%%%%%

\subsection{Characterizations for bar-joint frameworks with reflection symmetry}

We now give a combinatorial characterization of infinitesimally rigid
bar-joint frameworks with reflection symmetry $\mathcal{C}_s$ in the plane.
The following characterization of  rigid  frameworks with forced $\mathcal{C}_s$ symmetry
was already established in \cite{MT2,jkt}.

\begin{theorem}[Malestein and Theran~\cite{MT2,LT}, Jord{\'a}n et al. \cite{jkt}]
\label{thm:symmetry_reflection}
Let $\tau:\mathbb{Z}/2\mathbb{Z}\rightarrow {\cal C}_s$ be a faithful representation,
$(H,\psi)$ be a $\mathbb{Z}/2\mathbb{Z}$-gain graph,
and $\tilde{\bp}:V(H)\rightarrow \mathbb{R}^2$ be ${\cal C}_s$-regular.
Then $O_0(H,\psi,\tilde{\bp})$ is row independent if and only if
$(H,\psi)$ is $(2,3,1)$-gain-sparse.
\end{theorem}

We now show that independence of the other submatrix block is characterized by $(2,3,2)$-gain-sparsity.

\begin{theorem}
\label{thm:antisymmetry_reflection}
Let $\tau:\mathbb{Z}/2\mathbb{Z}\rightarrow {\cal C}_s$ be
a faithful representation,
$(H,\psi)$ be a $\mathbb{Z}/2\mathbb{Z}$-gain graph,
and $\tilde{\bp}:V(H)\rightarrow \mathbb{R}^2$ be ${\cal C}_s$-regular.
Then $O_1(H,\psi,\tilde{\bp})$ is row independent if and only if
$(H,\psi)$ is $(2,3,2)$-gain-sparse.
\end{theorem}
\begin{proof}
First we show that if $O_1(H,\psi,\tilde{\bp})$ is row independent then
$(H,\psi)$ is $(2,3,2)$-gain-sparse.
Suppose to the contrary that there exists a balanced $F\subseteq E(H)$ with $|F|> 2|V(F)|-3$.
Then, by Proposition~\ref{prop:switching} and Proposition~\ref{prop:induced_subgroup},
we may assume that $\psi(\te)=id$ for every $\te\in F$.
Then $O_1(H,\psi,\tilde{\bp})$ has a row dependency, because the submatrix of $O_1(H,\psi,\tilde{\bp})$ obtained
by deleting all rows in $O_1(H,\psi,\tilde{\bp})$ that do not correspond to edges in $F$
is a standard $2$-dimensional rigidity matrix with
$2|V(F)|$ columns and $|F|> 2|V(F)|-3$ edges. Suppose next that there exists an unbalanced subset $F$ of $E(H)$ with $|F|> 2|V(F)|-2$, and
assume wlog that the reflection is given by $\left(\begin{array}{c c} -1 & 0\\ 0 & 1 \end{array}\right)$.
Then $O_1(H,\psi,\tilde{\bp})$ again has a row dependency since it is easy to check that the infinitesimal translation $\tilde{\bmm}:V(H)\rightarrow \mathbb{R}^2$ defined by
$\tilde{\bmm}(\tilde{v})=\begin{pmatrix}1 \\ 0\end{pmatrix}$ for $\tilde{v}\in V(H)$ and the infinitesimal rotation $\tilde{\bmm}':V(H)\rightarrow \mathbb{R}^2$ defined by
$\tilde{\bmm}'(\tilde{v})=\begin{pmatrix}-(p_{\tilde{v}})_2\\ (p_{\tilde{v}})_1\end{pmatrix}$ for $\tilde{v}\in V(H)$ both lie in the kernel of $O_1(H,\psi,\tilde{\bp})$, and hence the kernel of
$O_1(H,\psi,\tilde{\bp})$ is of dimension at least $2$.

To prove that $(2,3,2)$-gain-sparsity is sufficient for $O_1(H,\psi,\tilde{\bp})$ to be row independent, we may employ induction on $|V(H)|$.
Suppose that $(H,\psi)$ is $(2,3,2)$-gain-sparse.
If $|V(H)|=1$, then $|E(H)|=0$, and there is nothing to prove.
If $|V(H)|>1$, we may assume that $|E(H)|=2|V(H)|-2$.
Combining Theorem~\ref{thm:construction} and Lemma~\ref{lem:extension},
we conclude that $O_1(H,\psi,\tilde{\bp})$ is row independent for a ${\cal C}_s$-regular $\tilde{\bp}$.
\end{proof}

It is easy to see that the same proof can be applied to show Theorem~\ref{thm:symmetry_reflection}
(which is the proof given in \cite{jkt}).

%Theorems~\ref{thm:symmetry_reflection} and \ref{thm:antisymmetry_reflection} imply the following combinatorial characterization ofinfinitesimally rigid frameworks with reflection symmetry.

\begin{theorem}
\label{thm:reflection}
Le $\tau:\mathbb{Z}/2\mathbb{Z}\rightarrow {\cal C}_s$ be a faithful representation,
$G$ be a $\mathbb{Z}/2\mathbb{Z}$-symmetric graph with $\theta:\mathbb{Z}/2\mathbb{Z}\rightarrow {\rm Aut}(G)$,
and $(G,\bp)$ be a ${\cal C}_s$-regular  framework with respect to $\theta$ and $\tau$.
Then the rank of $R(G,\bp)$ is equal to the sum of
the rank of ${\cal M}_{2,3,1}(H,\psi)$ and that of ${\cal M}_{2,3,2}(H,\psi)$,
where $(H,\psi)$ denotes  the quotient gain graph.
\end{theorem}
\begin{proof}
We may assume that $\bp$ is ${\cal C}_s$-generic.
By (\ref{rigblocks}) and Proposition~\ref{prop:cyclic},
we have $${\rm rank}\ R(G,\bp)={\rm rank}\ O_0(H,\psi,\tilde{\bp})+{\rm rank}\ O_1(H,\psi,\tilde{\bp})$$
for the quotient $\tilde{\bp}$ of $\bp$.
By Theorems~\ref{thm:symmetry_reflection} and \ref{thm:antisymmetry_reflection},
the rank of $O_j(H,\psi,\tilde{\bp})$ is equal to the rank of ${\cal M}_{2,3,1+j}(H,\psi)$ for $j=0,1$.
\end{proof}

Theorem~\ref{thm:reflection} shows how to compute the first-order degrees of freedom of $(G,\bp)$. However, if we are only interested
in checking infinitesimal rigidity, then we may use the following simpler condition.
\begin{corollary}
\label{cor:reflection}
Let $\tau:\mathbb{Z}/2\mathbb{Z} \rightarrow {\cal C}_s$ be a faithful representation,
$G$ be a $\mathbb{Z}/2\mathbb{Z}$-symmetric graph with $\theta:\mathbb{Z}/2\mathbb{Z}\rightarrow {\rm Aut}(G)$,
and $(G,\bp)$ be a ${\cal C}_s$-regular  framework with respect to $\theta$ and $\tau$.
Then $(G,\bp)$ is infinitesimally rigid if and only if
the quotient gain graph $(H,\psi)$ satisfies $|E(H)|\geq 2|V(H)|-1$ and contains a spanning subgraph $(H',\psi')$
which is $(2,3,2)$-gain-sparse with $|E(H')|=2|V(H')|-2$.
\end{corollary}
\begin{proof}
By Theorem~\ref{thm:reflection}, $(G,\bp)$ is infinitesimally rigid if and only if
$(H,\psi)$ contains two spanning subgraphs $(H_0,\psi_0)$ and $(H_1,\psi_1)$ such that
$(H_i,\psi_i)$ is $(2,3,i+1)$-sparse with $|E(H_i)|=2|V(H_i)|-(i+1)$ for $i=0,1$.
Observe that for a $(2,3,2)$-gain-sparse graph $(H_1,\psi_1)$ with $|E(H_1)|=2|V(H_1)|-2$,
adding a new edge to $H_1$ results in a $(2,3,1)$-gain-sparse graph $(H_0,\psi_0)$ with
$|E(H_0)|=2|V(H_0)|-1$.
This gives the result.
\end{proof}

For example, using Corollary~\ref{cor:reflection}, it is easy to verify that the framework shown in Figure~\ref{antisym}(a) is infinitesimally flexible (with an anti-symmetric infinitesimal flex): while the corresponding gain graph $(H,\psi)$ shown in Figure~\ref{csquotgr} is $(2,3,1)$-gain-sparse with $|E(H)|=6> 5=2|V(H)|-1$, it does not contain a spanning subgraph $(H',\psi')$
which is $(2,3,2)$-gain-sparse with $|E(H')|=2|V(H')|-2$. (Note that a loop violates $(2,3,2)$-gain sparsity.)

%%%%%%%%%%%%%%%%%%%%%%%%%%%%%%%%%%%%%%%%%%%%%%%%%%%%%%%%%%%%%%%%%%%%%%%%%%%%%%%%%%%%%%%%%%%%%%%%%%%%%%%%%%%%%%%%%%%%%%%%%%%%%%%%%%%%%%%%%%%%%%%%%%%%%%%%%%%%%%

\subsection{Characterizations for bar-joint frameworks with rotational symmetry}

We now discuss  combinatorial characterizations of infinitesimally rigid
frameworks with rotational symmetry ${\cal C}_k$ in the plane.
A characterization of the row independence  of $O_0(H,\psi,\tilde{p})$ was already established in \cite{MT3}.
(See also \cite{jkt} for a simpler proof).

\begin{theorem}[Malestein and Theran~\cite{MT3}]
\label{thm:symmetry_rotation}
Let $k\geq 2$,
$\tau:\mathbb{Z}/k\mathbb{Z}\rightarrow {\cal C}_k$ be a faithful representation,
$(H,\psi)$ be a $\mathbb{Z}/k\mathbb{Z}$-gain graph,
and $\tilde{\bp}:V(H)\rightarrow \mathbb{R}^2$ be ${\cal C}_k$-regular.
Then $O_0(H,\psi,\tilde{\bp})$ is row independent if and only if
$(H,\psi)$ is $(2,3,1)$-gain-sparse.
\end{theorem}

For frameworks with an arbitrary rotational symmetry $\mathcal{C}_k$, it is not as easy as for frameworks with reflection symmetry to
extend Theorem~\ref{thm:symmetry_rotation} to the other orbit matrices. However, the  following result holds for all rotational groups $\mathcal{C}_k$.

\begin{lemma}
\label{lem:rotation}
Let $k\geq 3$, $\tau:\mathbb{Z}/k\mathbb{Z}\rightarrow {\cal C}_k$
be a faithful representation,
$(H,\psi)$ be a $\mathbb{Z}/k\mathbb{Z}$-gain graph,
and $\tilde{\bp}:V(H)\rightarrow \mathbb{R}^2$ be ${\cal C}_k$-regular.
If $O_j(H,\psi,\tilde{\bp})$ is row independent,
then $(H,\psi)$ is $(2,3,0)$-gain-sparse.
Moreover, if  $j=1$ or $j=k-1$, then $O_j(H,\psi,\tilde{\bp})$ has a  kernel of dimension at least $1$, and
$(H,\psi)$ is $(2,3,1)$-gain-sparse.

Similarly, if $k=2$, then
the independence of $O_1(H,\psi,\tilde{\bp})$ implies that $(H,\psi)$ is $(2,3,2)$-gain-sparse.
\end{lemma}
\begin{proof}
Suppose that $O_j(H,\psi,\tilde{\bp})$ is row independent.
It is easy to see that $|F|\leq 2|V(F)|$ for any $F\subseteq E(H)$.

If $F$ is balanced, then, by Proposition~\ref{prop:switching} and Proposition~\ref{prop:induced_subgroup},
we may assume that $\psi(\te)=id$ for every $\te\in F$.
Then the submatrix of $O_j(H,\psi,\tilde{\bp})$ corresponding to the edges in $F$ is % equal to $R(H,\bp)$,
a standard $2$-dimensional rigidity matrix.
%of $(H,\bp)$.
Thus, $|F|\leq 2|V(F)|-3$ holds, and hence $(H,\psi)$ is $(2,3,0)$-gain-sparse.

Suppose further that $j=1$ or $j=k-1$.
We will show that $O_j(H,\psi,\tilde{\bp})$ always has a  kernel of dimension at least $1$.
To see this, recall that for any $\gamma\in \mathbb{Z}/k\mathbb{Z}$,
\begin{equation}
\label{eq:translation}
\tau(\gamma)\begin{pmatrix} 1 \\ \sqrt{-1} \end{pmatrix}=\omega^{\gamma}\begin{pmatrix} 1 \\ \sqrt{-1} \end{pmatrix} \qquad
\tau(\gamma)\begin{pmatrix} 1 \\ -\sqrt{-1} \end{pmatrix}=\bar{\omega}^{\gamma}\begin{pmatrix} 1 \\ -\sqrt{-1} \end{pmatrix}.
\end{equation}
where $\tau(\gamma)=\begin{pmatrix} \cos \gamma \theta & \sin \gamma \theta  \\
-\sin \gamma \theta & \cos \gamma \theta
 \end{pmatrix}$ and $\omega=e^{\sqrt{-1}\theta}$
with $\theta=\frac{2\pi}{k}$.

%(Shin: Here $i$ means $\sqrt{-1}$!! May we should replace all imaginary number $i$ with $\sqrt{-1}$ to avoid confusion?).BS: Yes - done!

If $j=1$, we define $\tilde{\bmm}:V(H)\rightarrow \mathbb{C}^2$ by
$\tilde{\bmm}(\tilde{v})=\begin{pmatrix}1 \\ \sqrt{-1}\end{pmatrix}$ for $\tilde{v}\in V(H)$.
Then, for any $\tilde{u},\tilde{v}\in V(H)$, we have
$\tilde{\bmm}(\tilde{u})-\bar{\omega}^{\gamma}\tau(\gamma)\tilde{\bmm}(\tilde{v})=\tilde{\bmm}(\tilde{u})-\bar{\omega}^{\gamma}\omega^{\gamma}\tilde{\bmm}(\tilde{v})=0$
by (\ref{eq:translation}),
which means that $\tilde{\bmm}$ is in the kernel of $O_1(H,\psi,\tilde{\bp})$ by (\ref{eq:j_inf_quot}).
Similarly, for $j=k-1$,
$\tilde{\bmm}:V(H)\rightarrow \mathbb{C}^2$ defined by
$\tilde{\bmm}(\tilde{v})=\begin{pmatrix}1 \\ -\sqrt{-1}\end{pmatrix}$ for $\tilde{v}\in V(H)$
is in the kernel of $O_{k-1}(H,\psi,\tilde{\bp})$.

Therefore, if $j=1$ or $j=k-1$, $|F|\leq 2|V(F)|-1$ must hold for any $F\subseteq E(H)$,
implying that $(H,\psi)$ is $(2,3,1)$-gain-sparse.

Similarly, if $k=2$, then the kernel of $O_1(H,\psi,\tilde{\bp})$ has dimension at least two
(which corresponds to the space of infinitesimal translations), and hence
$(H,\psi)$ is $(2,3,2)$-gain-sparse.
\end{proof}

Note that Lemma~\ref{lem:rotation} also shows how the space of infinitesimal translations is decomposed. This decomposition can also be read off from the character tables for the groups $\mathcal{C}_k$ (see \cite{altherz,bishop}, for example).

%%%%%%%%%%%%%%%%%%%%%%%%%%%%%%%%%%%%%%%%%%%%%%%%%%%%%%%%%%%%%

\subsubsection{Case of ${\cal C}_2$}

Combining Theorem~\ref{thm:construction}, Lemma~\ref{lem:extension},
Theorem~\ref{thm:symmetry_rotation}, and Lemma~\ref{lem:rotation},
we obtain the following characterization of infinitesimally rigid frameworks with ${\cal C}_2$ symmetry.
The proof is identical to that for ${\cal C}_s$ and hence is omitted.

\begin{theorem}
\label{thm:antisymmetry_rotation2}
Let $\tau:\mathbb{Z}/2\mathbb{Z} \rightarrow {\cal C}_2$ be
a faithful representation,
$(H,\psi)$ be a $\mathbb{Z}/2\mathbb{Z}$-gain graph,
and $\tilde{\bp}:V(H)\rightarrow \mathbb{R}^2$ be ${\cal C}_2$-regular.
Then $O_1(H,\psi,\tilde{\bp})$ is row independent if and only if
$(H,\psi)$ is $(2,3,2)$-gain-sparse.
\end{theorem}

\begin{theorem}
\label{thm:rotation2}
Let $\tau:\mathbb{Z}/2\mathbb{Z}\rightarrow {\cal C}_2$ be a faithful representation,
$G$ be a $\mathbb{Z}/2\mathbb{Z}$-symmetric graph with $\theta:\mathbb{Z}/2\mathbb{Z}\rightarrow {\rm Aut}(G)$,
and $(G,\bp)$ be a ${\cal C}_2$-regular  framework with respect to $\theta$ and $\tau$.
Then the rank of $R(G,\bp)$ is equal to the sum of
the rank of ${\cal M}_{2,3,1}(H,\psi)$ and that of ${\cal M}_{2,3,2}(H,\psi)$,
where $(H,\psi)$ denotes  the quotient gain graph.
\end{theorem}

\begin{corollary}
\label{cor:rot2}
Let $\tau:\mathbb{Z}/2\mathbb{Z} \rightarrow {\cal C}_2$ be a faithful representation,
$G$ be a $\mathbb{Z}/2\mathbb{Z}$-symmetric graph with $\theta:\mathbb{Z}/2\mathbb{Z}\rightarrow {\rm Aut}(G)$,
and $(G,\bp)$ be a ${\cal C}_2$-regular  framework with respect to $\theta$ and $\tau$.
Then $(G,\bp)$ is infinitesimally rigid if and only if
the quotient gain graph $(H,\psi)$ satisfies $|E(H)|\geq 2|V(H)|-1$ and contains a spanning subgraph $(H',\psi')$
which is $(2,3,2)$-gain-sparse with $|E(H')|=2|V(H')|-2$.
\end{corollary}

%\begin{corollary}
%\label{cor:rotation2}
%Let $\Gamma$ be a cyclic group of order $2$,  $\tau:\Gamma\rightarrow {\cal C}_2$ be a faithful representation,
%$G$ be a $\Gamma$-symmetric graph with $\theta:\Gamma\rightarrow G$,
%and $(G,\bp)$ be a $\Gamma$-symmetric framework with respect to $\theta$ and $\tau$.
%Then $(G,\bp)$ is infinitesimally rigid if and only if
%the quotient gain graph $(H,\psi)$ satisfies $|E(H)|\geq 2|V(H)|-1$ and contains a subgraph $(H',\psi')$
%which is $(2,3,2)$-gain sparse with $|E(H)|=2|V(H)|-2$.
%\end{corollary}

%%%%%%%%%%%%%%%%%%%%%%%%%%%%%%%%%%%%%%%%%%%%%%%%%%%%%%%%%%%%%%%%%%%%

\subsubsection{Case of ${\cal C}_3$}
%\begin{theorem}
%\label{thm:antisymmetry_rotation3}
%Let $\Gamma$ be a cyclic group of order three along along with
%a faithful representation $\tau:\Gamma\rightarrow {\cal C}_2$,
%$(H,\psi)$ be a $\Gamma$-gain graph, and $p:V(H)\rightarrow \mathbb{R}^2$ be $\Gamma$-regular.
%Then, for $10\leq j\leq 2$, $O_j(H,\psi,p)$ is row independent if and only if
%$(H,\psi)$ is $(2,3,1)$-gain-sparse.
%\end{theorem}

\begin{theorem}
\label{thm:rotation3}
Let $\tau:\mathbb{Z}/3\mathbb{Z}\rightarrow {\cal C}_3$ be a faithful representation,
$G$ be a $\mathbb{Z}/3\mathbb{Z}$-symmetric graph with $\theta:\mathbb{Z}/3\mathbb{Z}\rightarrow {\rm Aut}(G)$,
and $(G,\bp)$ be a ${\cal C}_3$-regular framework with respect to $\theta$ and $\tau$.
Then the rank of $R(G,\bp)$ is equal to three times
the rank of ${\cal M}_{2,3,1}(H,\psi)$,
where $(H,\psi)$ denotes  the quotient gain graph.
\end{theorem}
\begin{proof}
We show that for each $j=1,2$, $O_j(H,\psi,\tilde{\bp})$ is row independent if and only if
$(H,\psi)$ is $(2,3,1)$-gain-sparse.
This implies the statement, by Proposition~\ref{prop:cyclic}  and Theorem~\ref{thm:symmetry_rotation}.

By Lemma~\ref{lem:rotation}, if $O_j(H,\psi,\tilde{\bp})$ is row independent,
$(H,\psi)$ is $(2,3,1)$-gain-sparse.

We show the converse direction by induction on $|V(H)|$.
Suppose $(H,\psi)$ is $(2,3,1)$-gain-sparse.
Proposition~\ref{lem:zero_loop} implies that an unbalanced loop is a zero loop in $O_j(H,\psi,\tilde{\bp})$ only if
the underlying group contains a subgroup isomorphic to $\mathbb{Z}/2\mathbb{Z}$.
Hence, in this case, a loop cannot be a zero loop, which in particular implies that
$O_j(H,\psi,\tilde{\bp})$ is row independent when $|V(H)|=1$.
If $|V(H)|>1$, then we can construct $\tilde{\bp}:V(H)\rightarrow \mathbb{R}^2$ such that
$(H,\psi,\tilde{\bp})$ is row independent by induction, using
Theorem~\ref{thm:construction} and Lemma~\ref{lem:extension}.
\end{proof}

As a corollary, we obtain the following characterization given in \cite{BS3}.
\begin{corollary}[Schulze \cite{BS3}]
\label{cor:rotation3}
Let $\tau:\mathbb{Z}/3\mathbb{Z}\rightarrow {\cal C}_3$ be a faithful representation,
$G$ be a $\mathbb{Z}/3\mathbb{Z}$-symmetric graph with $\theta:\mathbb{Z}/3\mathbb{Z}\rightarrow {\rm Aut}(G)$,
and $(G,\bp)$ be a ${\cal C}_3$-regular  framework with respect to $\theta$ and $\tau$.
Then $(G,\bp)$ is infinitesimally rigid if and only if
the quotient gain graph $(H,\psi)$ contains a spanning subgraph $(H',\psi')$
which is $(2,3,1)$-gain sparse with $|E(H')|=2|V(H')|-1$.
\end{corollary}

%%%%%%%%%%%%%%%%%%%%%%%%%%%%%%%%%%%%%%%%%%%%%%%%%%%%%%%%%%%%%%%%%%55

%%%%%%%%%%%%%%%%%%%%%%%%%%%%%%%%%%%%%%%%%%%%%%%%%%%%%%%%%%%%%%%%%%55

\subsubsection{Case of ${\cal C}_k$ with $k\geq 4$}

%For $k\geq 4$, it truns out that $O_j(H<\psi,\bp)$ has more complex combinatorial properties.
The following lemma gives a necessary condition for the row independence of $O_j(H,\psi,\tilde{\bp})$ for even $k$,
which  is  stronger than the one given in Lemma~\ref{lem:rotation}.

\begin{lemma}
\label{lem:neccessary_even}
Let $k\geq 4$,
$\tau:\mathbb{Z}/k\mathbb{Z}\rightarrow {\cal C}_k$ be a faithful representation,
 $(H,\psi)$ be a $\mathbb{Z}/k\mathbb{Z}$-gain graph,
$\tilde{\bp}:V(H)\rightarrow \mathbb{R}^2$ be ${\cal C}_k$-regular,
and $j$ be an odd integer with $1\leq j<k$.
If $O_{j}(H,\psi,\tilde{\bp})$ is row independent,
then $F$ is $(2,3,2)$-gain-sparse for any $F\subseteq E(H)$ such that $\langle F\rangle$ is isomorphic to
$\mathbb{Z}/2\mathbb{Z}$.
\end{lemma}
\begin{proof}
Let $\omega=e^{\frac{2\pi \sqrt{-1}}{k}}$.
Since $\langle F\rangle$ is isomorphic to $\mathbb{Z}/2\mathbb{Z}$,
$\langle F\rangle$ consists of $\{0,k/2\}$.
Let $h:\{0,k/2\}\rightarrow \mathbb{Z}/2\mathbb{Z}$ be the isomorphism.

By Proposition~\ref{prop:switching} and Proposition~\ref{prop:induced_subgroup}, we may assume that
$\psi(\te)\in \{0,k/2\}$ for all $\te\in F$, and hence
we can define a gain function $\psi':F\rightarrow \mathbb{Z}/2\mathbb{Z}$ by
$\psi'(\te)=h(\psi(\te))$ for $\te\in F$.
Also, we can define $\tau':\mathbb{Z}/2\mathbb{Z}\rightarrow {\cal C}_2$
by $\tau'=\tau\circ h^{-1}$.

Observe that $\omega^{jk/2}=\omega^{k/2}=-1$ if $j$ is odd,
which implies $\omega^{j\psi(\te)}=(-1)^{\psi'(\te)}$ for $\te\in F$.
Therefore, for $\te=(\tilde{u},\tilde{v})\in F$, we have
\begin{align*}
\tilde{\bp}(\tilde{u})-\tau(\psi(\te))\tilde{\bp}(\tilde{v})&=\tilde{\bp}(\tilde{u})-\tau'(\psi'(\te))\tilde{\bp}(\tilde{v}) \\
\omega^{j\psi(\te)}(\tilde{\bp}(\tilde{v})-\tau(\psi(\te))^{-1}\tilde{\bp}(\tilde{u}))&=(-1)^{\psi'(\te)}(\tilde{\bp}(\tilde{v})-\tau'(\psi'(\te))^{-1}\tilde{\bp}(\tilde{u})).
\end{align*}
In other words, we have $O_j(H[F],\psi, \tilde{\bp})=O_1(H[F],\psi',\tilde{\bp})$,
where $H[F]$ is the subgraph of $H$ induced by the edge set $F$.
Since $(H[F],\psi')$ is a $\mathbb{Z}/2\mathbb{Z}$-gain graph along with a faithful representation
$\tau':\mathbb{Z}/2\mathbb{Z}\rightarrow {\cal C}_2$,
$F$ is $(2,3,2)$-gain-sparse by Lemma~\ref{lem:rotation}.
\end{proof}

It follows from this lemma that if $k$ is even, then
there is  a ${\cal C}_k$-generic  framework $(G,\bp)$ such that
the underlying graph is 2-rigid (i.e., generically rigid in the plane)
but $(G,\bp)$ is not infinitesimally rigid.
However, we still conjecture that Laman's condition characterizes infinitesimal rigidity for odd $k$.
\begin{conj}\label{con:oddk}
Let ${\cal C}_k$ be the group generated by a $k$-fold rotation in the plane, where $k$ is odd.
Let $(G,\bp)$ be a ${\cal C}_k$-generic   framework.
Then $(G,\bp)$ is infinitesimally rigid if and only if $G$ is 2-rigid.
\end{conj}

One possible approach for proving this conjecture is to develop a constructive
characterization of 2-rigid $\mathbb{Z}/k\mathbb{Z}$-symmetric graphs.
Since there is a one-to-one correspondence between $\mathbb{Z}/k\mathbb{Z}$-symmetric graphs
and $\mathbb{Z}/k\mathbb{Z}$-gain graphs
(up to the choices of representative vertices),
our task is to extend Theorem~\ref{thm:construction}.
In the following, we make several observations concerning Conjecture~\ref{con:oddk}.

\begin{theorem}
\label{thm:sparsity_relation}
Let $G$ be a $\mathbb{Z}/k\mathbb{Z}$-symmetric graph with odd $k\geq 3$
and $(H,\psi)$ be its quotient $\mathbb{Z}/k\mathbb{Z}$-gain graph.
If $(H,\psi)$ is $(2,3,1)$-gain-sparse, then $G$ is 2-independent.
\end{theorem}
\begin{proof}
By Theorem~\ref{thm:construction}, $(H,\psi)$ can be constructed
from a gain graph with one vertex with a loop with non-identity label by
0-extensions, 1-extensions, and loop-1-extensions.
Since $k$ is odd, Proposition~\ref{lem:zero_loop} implies that
a zero-loop does not occur.
Therefore, by Lemma~\ref{lem:extension}, there is an injective $\bp:V(G)\rightarrow \mathbb{R}^2$ such that
$(G,\bp)$ is ${\cal C}_k$-symmetric and $R(G,\bp)$ is row independent.
The row independence of $R(G,\bp)$ implies that $G$ is 2-independent.
\end{proof}

Theorem~\ref{thm:sparsity_relation} says that the covering graph of any $(2,3,1)$-gain-tight graph $(H,\psi)$ is
$2$-independent if $k$ is odd.
Since the covering graph $G$ has $k|E(H)|$ edges, which is equal to $k(2|V(H)|-1)=2|V(G)|-k$,
$G$ cannot be 2-rigid if $k>3$.
The next step is hence to investigate
which new edges we should add so that the covering graph is 2-rigid.
This question turns out to be complicated, as the following examples illustrate.

%\begin{figure}[t]
%\centering
%\begin{minipage}{0.4\textwidth}
%\centering
%\includegraphics[scale=0.6]{balanced_circuit.pdf}
%\par
%(a)
%\end{minipage}
%\begin{minipage}{0.4\textwidth}
%\centering
%\includegraphics[scale=1]{balanced_circuit_q.pdf}
%\par
%(b)
%\end{minipage}
%\begin{minipage}{0.4\textwidth}
%\centering
%\includegraphics[scale=0.6]{unbalanced_circuit.pdf}
%\par
%(c)
%\end{minipage}
%\begin{minipage}{0.4\textwidth}
%\centering
%\includegraphics[scale=1]{unbalanced_circuit_q.pdf}
%\par
%(d)
%\end{minipage}
%\caption{A balanced circuit (a) and its underlying quotient gain graph (b).
%Note that we may assume that the label of each edge is identity by Proposition~\ref{prop:induced_subgroup}.
%An unbalanced circuit (c) and its underlying quotient graph (d).}
%\label{fig:circuit}
%\end{figure}

\begin{figure}[htp]
\begin{center}
\begin{tikzpicture}[very thick,scale=1]
\tikzstyle{every node}=[circle, draw=black, fill=white, inner sep=0pt, minimum width=5pt];
\node (p1) at (60:0.5cm) {};
\node (p2) at (132:0.5cm) {};
\node (p3) at (204:0.5cm) {};
\node (p4) at (276:0.5cm) {};
\node (p5) at (348:0.5cm) {};

\node (a1) at (60:1.5cm) {};
\node (a2) at (132:1.5cm) {};
\node (a3) at (204:1.5cm) {};
\node (a4) at (276:1.5cm) {};
\node (a5) at (348:1.5cm) {};

\node (r1) at (40:1.15cm) {};
\node (r2) at (112:1.15cm) {};
\node (r3) at (184:1.15cm) {};
\node (r4) at (256:1.15cm) {};
\node (r5) at (328:1.15cm) {};

\node (l1) at (85:1.3cm) {};
\node (l2) at (157:1.3cm) {};
\node (l3) at (229:1.3cm) {};
\node (l4) at (301:1.3cm) {};
\node (l5) at (13:1.3cm) {};

\draw(p1)--(a1);
\draw(p1)--(r1);
\draw(p1)--(l1);
\draw(l1)--(r1);
\draw(l1)--(a1);
\draw(a1)--(r1);

\draw(p2)--(a2);
\draw(p2)--(r2);
\draw(p2)--(l2);
\draw(l2)--(r2);
\draw(l2)--(a2);
\draw(a2)--(r2);

\draw(p3)--(a3);
\draw(p3)--(r3);
\draw(p3)--(l3);
\draw(l3)--(r3);
\draw(l3)--(a3);
\draw(a3)--(r3);

\draw(p4)--(a4);
\draw(p4)--(r4);
\draw(p4)--(l4);
\draw(l4)--(r4);
\draw(l4)--(a4);
\draw(a4)--(r4);

\draw(p5)--(a5);
\draw(p5)--(r5);
\draw(p5)--(l5);
\draw(l5)--(r5);
\draw(l5)--(a5);
\draw(a5)--(r5);

\node [draw=white, fill=white] (a) at (0,-2)  {(a)};

\path ([xshift=3cm,yshift=-1cm]p1) node (pp1){};
\path ([xshift=3cm,yshift=-1cm]a1) node (aa1){};
\path ([xshift=3cm,yshift=-1cm]r1) node (rr1){};
\path ([xshift=3cm,yshift=-1cm]l1) node (ll1){};

%\node (p1) at (60:0.5cm) {};
%\node (a1) at (60:1.5cm) {};
%\node (r1) at (40:1.15cm) {};
%\node (l1) at (85:1.3cm) {};

\draw(pp1)--(aa1);
\draw(pp1)--(rr1);
\draw(pp1)--(ll1);
\draw(ll1)--(rr1);
\draw(ll1)--(aa1);
\draw(aa1)--(rr1);

%\node [draw=white, fill=white] (a) at (0,-0.1)  {$\quad$};

\node [draw=white, fill=white] (a) at (3.3,-2)  {(b)};
\end{tikzpicture}
\hspace{1cm}
\begin{tikzpicture}[very thick,scale=1]
\tikzstyle{every node}=[circle, draw=black, fill=white, inner sep=0pt, minimum width=5pt];
\node (p1) at (60:1.2cm) {};
\node (p2) at (132:1.2cm) {};
\node (p3) at (204:1.2cm) {};
\node (p4) at (276:1.2cm) {};
\node (p5) at (348:1.2cm) {};

\node (a1) at (30:1.5cm) {};
\node (a2) at (102:1.5cm) {};
\node (a3) at (174:1.5cm) {};
\node (a4) at (246:1.5cm) {};
\node (a5) at (318:1.5cm) {};

\node (r1) at (13:0.6cm) {};
\node (r2) at (85:0.6cm) {};
\node (r3) at (157:0.6cm) {};
\node (r4) at (229:0.6cm) {};
\node (r5) at (301:0.6cm) {};

\node (l1) at (85:0.6cm) {};
\node (l2) at (157:0.6cm) {};
\node (l3) at (229:0.6cm) {};
\node (l4) at (301:0.6cm) {};
\node (l5) at (13:0.6cm) {};

\draw(p1)--(a1);
\draw(p1)--(r1);
\draw(p1)--(l1);
\draw(l1)--(r1);
\draw(l1)--(a1);
\draw(a1)--(r1);

\draw(p2)--(a2);
\draw(p2)--(r2);
\draw(p2)--(l2);
\draw(l2)--(r2);
\draw(l2)--(a2);
\draw(a2)--(r2);

\draw(p3)--(a3);
\draw(p3)--(r3);
\draw(p3)--(l3);
\draw(l3)--(r3);
\draw(l3)--(a3);
\draw(a3)--(r3);

\draw(p4)--(a4);
\draw(p4)--(r4);
\draw(p4)--(l4);
\draw(l4)--(r4);
\draw(l4)--(a4);
\draw(a4)--(r4);

\draw(p5)--(a5);
\draw(p5)--(r5);
\draw(p5)--(l5);
\draw(l5)--(r5);
\draw(l5)--(a5);
\draw(a5)--(r5);

\node [draw=white, fill=white] (a) at (0,-2)  {(c)};

\path ([xshift=3cm,yshift=-0.5cm]p1) node (pp1){};
\path ([xshift=3cm,yshift=-0.5cm]a1) node (aa1){};
\path ([xshift=3cm,yshift=-0.5cm]r1) node (rr1){};
%\node (p1) at (60:1.2cm) {};
%\node (a1) at (30:1.5cm) {};
%\node (r1) at (13:0.6cm) {};

\draw(pp1)--(aa1);
\draw(pp1)--(rr1);
\draw(aa1)--(rr1);

\path
(rr1) edge [<-,bend left=22] (pp1);
\path
(rr1) edge [<-,bend right=22] (aa1);

 \path
(rr1) edge [loop below,->, >=stealth,shorten >=2pt,looseness=26] (rr1);

   \node [draw=white, fill=white] (a) at (4.4,-0.2)  {$\gamma$};
   \node [draw=white, fill=white] (a) at (3.2,-0.7)  {$\gamma$};
   \node [draw=white, fill=white] (a) at (3.2,0.2)  {$\gamma$};

   \node [draw=white, fill=white] (a) at (3.7,-2)  {(d)};
\end{tikzpicture}

\end{center}
\vspace{-0.2cm}
\caption{A balanced circuit (b) and its corresponding covering graph  (a).
Note that we may assume that the label of each edge is the identity, by Proposition~\ref{prop:induced_subgroup}.
An unbalanced circuit (d) and its corresponding covering graph (c).}
\label{fig:circuit}\end{figure}
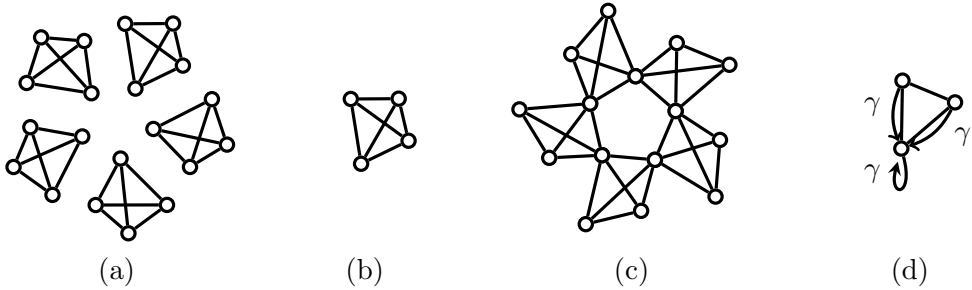

Let us consider a $\mathbb{Z}/k\mathbb{Z}$-gain graph $(H,\psi)$
which consists of a $(2,3,1)$-gain-tight graph $(H',\psi)$ together with one additional edge $\te$
(i.e., $H=H'+\te$).
The covering graph of $H$ and the covering map is denoted by $G$ and $c$, respectively.
It follows easily from Proposition~\ref{prop:induced_subgroup}
that if $(H,\psi)$ contains an edge set $F$ such that
(i) $F$ is balanced and (ii) $|F|>2|V(F)|-3$,
then $c^{-1}(F)$ consists of $k$ vertex-disjoint 2-dependent sets.
A minimal edge set $F$ satisfying (i) and (ii) is called a {\em balanced circuit}.
See Figure~\ref{fig:circuit}(a)(b) for an example.

There is another obstacle.
Suppose that there is an edge subset $F$
such that (i') $F$ is unbalanced, (ii') $|F|>2|V(F)|-1$, and
(iii') there are a vertex $\tilde{v}\in V(F)$, an element $\gamma\in \mathbb{Z}/k\mathbb{Z}$,
and a labeling function $\psi':E(H)\rightarrow \mathbb{Z}/k\mathbb{Z}$ equivalent to $\psi$ such that
$\psi'(\te)=id$ for every $\te\in F$ not incident to $\tilde{v}$,
and $\psi'(\te)\in \{id, \gamma\}$ for every $\te\in F$ directed to $\tilde{v}$
(assuming that every edge incident to $\tilde{v}$ is directed to $\tilde{v}$).
See also Figure~\ref{fig:circuit}(c)(d).
Then it can  easily be checked
that the covering graph $c^{-1}(F)$ is the union of $k$ edge-disjoint 2-dependent sets.
A minimal edge set $F$ satisfying (i')(ii')(iii') is called an {\em unbalanced circuit}.

Consequently, if $H=H'+\te$  contains an unbalanced circuit or a balanced circuit,
the covering graph $G$ contains $k$ edge-disjoint 2-dependent sets, which means that
no edge of $c^{-1}(\te)$ increases the rank of the 2-rigidity matroid of the covering graph.

\section{Extensions} \label{sec:ext}

We finish by making some further comments about `anti-symmetric' orbit rigidity matrices and their applications and by outlining some directions for future developments.

\subsection{Bar-joint frameworks in higher dimensions}

As we mentioned in the introduction, it is a key open problem in rigidity theory to find a combinatorial characterization of infinitesimally rigid generic bar-joint frameworks (without symmetry) in dimensions $3$ and higher. Therefore, we restricted attention to two-dimensional symmetric frameworks in Sections~\ref{sec:gain_sparsity} and \ref{sec:characterization}. However, note that we showed in Section~\ref{sec:block} how to construct anti-symmetric orbit rigidity matrices for a symmetric framework in an \emph{arbitrary} dimension $d$. 

Each of these anti-symmetric orbit rigidity matrices gives rise to an independent set of necessary  conditions for the framework to be infinitesimally rigid in $\mathbb{R}^d$. Analogously to the conditions derived for various symmetric two-dimensional frameworks in Section~\ref{sec:characterization}, these conditions can of course be expressed as gain-sparsity conditions for the corresponding quotient gain-graph. However, to state these conditions, we need to compute the dimension of the space of trivial infinitesimal motions which are symmetric with respect to the given irreducible representation. In dimension $3$, the dimensions of these spaces can be read off directly from the character tables of the group (see \cite{altherz,bishop}, for example); for dimensions $4$ and higher, one needs to compute these dimensions for each individual group. This can be done in a similar way as in the proof of Lemma~\ref{lem:rotation}, for example (see also \cite{BS2}).

Finally, note that due the simplicity of its entries and its straightforward construction, each of the orbit rigidity matrices of a given $d$-dimensional framework allows a quick analysis of its row or column dependencies, and hence provides a powerful tool  for the detection of infinitesimal motions and self-stresses which exhibit the symmetries of the corresponding irreducible representation and which cannot be found by checking the corresponding gain-sparsity counts. 

\subsection{Non-Abelian groups}

In Section~\ref{sec:block} we showed how to construct anti-symmetric orbit rigidity matrices for frameworks with any Abelian point group symmetry in an arbitrary dimension. The key problem to extend these constructions to frameworks with non-Abelian point group symmetries is that each non-Abelian point group has an irreducible representation which is of dimension at least $2$, and an infinitesimal motion which is symmetric with respect to such a higher-dimensional representation is not uniquely determined by the velocity vectors assigned to the vertices in the quotient gain-graph. Therefore, the entries of an orbit rigidity matrix corresponding to such a representation (as well as the underlying combinatorial structure for such an orbit matrix) are more complicated. It remains open how to extend our methods and results to frameworks with non-Abelian point group symmetries.

\subsection{Group actions which are not free on the vertex set}

Throughout this paper, we assumed that the group $\Gamma$ acts freely on the vertex set of the graph $G$. While in principle we do not expect any major new complications to arise if we allowed $\Gamma$ to act non-freely on the vertices of $G$, the structures of the orbit rigidity matrices and the corresponding gain-sparsity counts would need to be adjusted accordingly and would become significantly less clear and transparent (see also \cite{BSWW}).

For example, suppose a joint $p_i$ of a two-dimensional $\mathcal{C}_s$-symmetric framework $(G,\bp)$ is `fixed' by the reflection $s$ in $\mathcal{C}_s$, i.e., we have $\tau(s)(p_i)=p_i$. Then $p_i$ contributes only one column to the fully symmetric orbit rigidity matrix of $(G,\bp)$ (as $p_i$ has only a one-dimensional space of fully symmetric displacement vectors: the space of all vectors which lie along the mirror line of $s$) and only one column to the anti-symmetric orbit rigidity matrix of $(G,\bp)$ (as $p_i$ has also only  a one-dimensional space of anti-symmetric displacement vectors: the space of all vectors which lie perpendicular to the mirror line of $s$).  Similarly, if $p_i$ is a joint of a two-dimensional $\mathcal{C}_2$-symmetric framework $(G,\bp)$ which is `fixed' by the half-turn $C_2$, then $p_i$ would contribute no column to the fully symmetric orbit rigidity matrix of $(G,\bp)$ (as $p_i$ has no fully symmetric displacement vectors) and two columns to the anti-symmetric orbit rigidity matrix of $(G,\bp)$ (as $p_i$ has  a two-dimensional space of anti-symmetric displacement vectors).

Due to these modifications to the structures and entries of the orbit rigidity matrices, the constructions of these matrices and the proofs for the combinatorial characterizations of $\Gamma$-generic infinitesimally rigid frameworks in the plane will become significantly more messy.

\subsection{Extensions to body-bar and body-hinge frameworks}

One now standard extension of bar-joint frameworks are the body-bar frameworks \cite{W1,Tay84}. These form a special class of bar-joint frameworks, which have many important practical applications in fields such as engineering, robotics or biochemistry.  Note that while a combinatorial characterization of $3$- or higher-dimensional bar-joint frameworks has not yet been found, rigid generic body-bar frameworks (without symmetry) were characterized in all dimensions by Tay \cite{Tay84}.

 In \cite{schtan}, we extend our tools and methods to $d$-dimensional body-bar frameworks with Abelian point group symmetries by giving a description of symmetric body-bar frameworks in terms of the Grassmann-Cayley algebra. Moreover, we establish combinatorial characterizations of body-bar frameworks which are generic with respect to a point group of the form $\mathbb{Z}/2\mathbb{Z}\times \dots \times \mathbb{Z}/2\mathbb{Z}$ using Dowling matroids.

Finally, in \cite{schtan} we also extend our methods and results to body-hinge frameworks, i.e., to structures which consist of rigid bodies that are connected, in pairs, by revolute hinges along assigned lines. This is an important step  towards applying our results to the rigidity and flexibility analysis of certain physical structures like robotic linkages or biomolecules.

\providecommand{\bysame}{\leavevmode\hbox to3em{\hrulefill}\thinspace}
\providecommand{\MR}{\relax\ifhmode\unskip\space\fi MR }
\providecommand{\MRhref}[2]{%
  \href{http://www.ams.org/mathscinet-getitem?mr=#1}{#2}
}
\providecommand{\href}[2]{#2}


\begin{thebibliography}{10}

 \bibitem{altherz}
S. L. Altmann and P. Herzig, \emph{Point-Group Theory Tables}, Clarendon
Press, Oxford, 1994

\bibitem{asiroth}
L.~Asimov and B.~Roth, \emph{The {R}igidity of {G}raphs}, AMS \textbf{245}
  (1978), 279--289.

\bibitem{bishop}
D.M. Bishop, \emph{Group {T}heory and {C}hemistry}, Clarendon Press, Oxford,
  1973.

 % \bibitem{bost} C. Borcea and I. Streinu, \emph{Minimally rigid periodic graphs}, Bulletin of the LMS, \textbf{43} (2011), No. 6, 1093--1103.
 

\bibitem{cfgsw}
R.~Connelly, P.W. Fowler, S.D. Guest, B.~Schulze, and W.~Whiteley, \emph{When
  is a symmetric pin-jointed framework isostatic?}, International Journal of
  Solids and Structures \textbf{46} (2009), 762--773.


\bibitem{FGsymmax}
P.W. Fowler and S.D. Guest, \emph{A symmetry extension of {M}axwell's rule for
  rigidity of frames}, International Journal of Solids and Structures
  \textbf{37} (2000), 1793--1804.


%\bibitem{gsw}
%S.D. Guest, B.~Schulze, and W.~Whiteley, \emph{When is a symmetric body-bar
%  structure isostatic?},  International Journal of Solids and
%  Structures \textbf{47} (2010), 2745--2754.

%\bibitem{graver}J. Graver, \emph{Counting on Frameworks}, Mathematical Association of America, 2001.

%\bibitem{gss}
%J.E. Graver, B.~Servatius, and H.~Servatius, \emph{Combinatorial {R}igidity},
%  Graduate Studies in Mathematics, AMS, 1993.


 \bibitem{jkt}
T. Jordan, V. Kaszanitzky and S. Tanigawa, \emph{Gain-sparsity and symmetry-forced rigidity in the plane}, The EGRES technical report, TR-2012-17.


\bibitem{KG2}
R.D. Kangwai and S.D. Guest, \emph{Symmetry-adapted equilibrium matrices}, International Journal of
  Solids and Structures \textbf{37} (2000), 1525--1548.


\bibitem{Lamanbib}
G. Laman, \emph{On graphs and rigidity of plane skeletal structures}, J. Engrg.
Math. \textbf{4} (1970), 331-340


%\bibitem{lovyem} L. Lov\'asz and Y. Yemini, \emph{On generic rigidity in the plane}, SIAM J. Alg. Disc. Methods 3 (1982), 91--98.



\bibitem{MT3} J.~Malestein and L.~Theran, \emph{Generic rigidity of frameworks with orientation-preserving crystallographic symmetry},
preprint, arXiv:1108.2518, 2011


 \bibitem{MT2}
\bysame, \emph{Generic rigidity of reflection frameworks}, preprint, arXiv:1203.2276, 2012.

\bibitem{MT1}
\bysame, \emph{Generic rigidity with forced symmetry and sparse colored graphs}, preprint, arXiv:1203.0772, 2012.


%\bibitem{nr} T. Nixon and E. Ross, \emph{Periodic rigidity on a variable torus using inductive constructions}, preprint, arXiv:1204.1349, 2012.


\bibitem{owen}
J.C. Owen and S.C. Power, \emph{Frameworks, symmetry and rigidity}, Int. J. Comput. Geom. Appl. \textbf{20} (2010), 723--750.

 \bibitem {ER} E. Ross, \emph{Geometric and combinatorial rigidity of periodic frameworks as graphs on the torus}, Ph. D. thesis, York University, Toronto, May 2011.

 %\bibitem{rsw} E. Ross, B. Schulze, and W. Whiteley, \emph{Finite motions from periodic frameworks with added symmetry}, Int. J. of Solids and Structures, \textbf{48} (2011), 1711 �V1729.

\bibitem{BS2}
B.~Schulze, \emph{Block-diagonalized rigidity matrices of symmetric frameworks and
  applications}, Contributions to Algebra and Geometry \textbf{51} (2010), No. 2, 427--466.

\bibitem{BS1}
\bysame, \emph{Injective and non-injective realizations with symmetry},
  Contributions to Discrete Mathematics \textbf{5} (2010), 59--89.

\bibitem{BS3}
\bysame, \emph{Symmetric versions of {L}aman's {T}heorem}, Discrete and Computational Geometry \textbf{44} (2010), No. 4, 946--972.

\bibitem{BS4} \bysame, \emph{Symmetric Laman theorems for the groups $C_2$ and $C_s$},
The Electronic Journal of Combinatorics \textbf{17} (2010), No. 1, R154, 1--61.

\bibitem{schtan} B.~Schulze and S.~Tanigawa, \emph{Infinitesimal rigidity of symmetric body-bar and body-hinge frameworks}, preprint
(see Sections 7 and 8 in arXiv:1308.6380), 2013

\bibitem{BSWW} B.~Schulze and W.~Whiteley, \emph{The orbit rigidity matrix of a symmetric framework},
Discrete and Computational Geometry \textbf{46} (2011), No. 3, 561--598.


\bibitem{tan}
S. Tanigawa, \emph{Matroids of gain graphs in applied discrete geometry}, arXiv:1207.3601, 2012.


\bibitem{Tay84} T.-S. Tay, \emph{Rigidity of multi-graphs, linking rigid bodies in $n$-space}, J. Comb. Theory, B \textbf{36} (1984), 95--112.


 %\bibitem{TW1} T.-S. Tay and W. Whiteley,  \emph{Recent advances in generic rigidity of structures}, Structural Topology \textbf{9} (1985), 31--38.
 
 \bibitem{LT}
L.~Theran, \emph{Henneberg constructions and covers of cone-Laman graphs}, preprint, arXiv:1204.0503,
  2012.


\bibitem{W1}
W. Whiteley, \emph{Some {M}atroids from {D}iscrete {A}pplied {G}eometry},
  Contemporary Mathematics, AMS \textbf{197} (1996), 171--311.

%\bibitem{zaslavsky1989biased}
%T.~Zaslavsky,
%\emph{Biased graphs "{I}": Bias, balance, and gains},
% J.~Combin.~Theory Ser.~B, \textbf{47} (1989), 32--52.

\bibitem{zaslavsky1991biased}
\bysame,
\emph{Biased graphs "{II}": The three matroids},
J.~Combin.~Theory Ser.~B, \textbf{51} (1991), 46--72.


\end{thebibliography}
\end{document}